\documentclass{amsart}
\usepackage{amsmath,amssymb,latexsym,cancel,rotating}
\usepackage{graphicx,amssymb,mathrsfs,amsmath,color,fancyhdr,amsthm}
\usepackage[all]{xy}
\usepackage{bm}
\usepackage{tikz}
\usetikzlibrary{calc}
\usetikzlibrary{decorations}
\newtheorem{theorem}{Theorem}[section]
\newtheorem{lemma}[theorem]{Lemma}
\newtheorem{corollary}[theorem]{Corollary}
\newtheorem{definition}[theorem]{Definition}
\newtheorem{proposition}[theorem]{Proposition}

\newtheorem{remark}[theorem]{Remark}

\def\lan{\langle}    \def\ran{\rangle}

\def\az{\alpha}
\def\bz{\beta}
\def\gz{\gamma}  \def\ggz{\Gamma}
\def\dz{\delta}  
\def\oz{\omega}  \def\hoz{{\hat{\omega}}}  \def\ooz{\Omega}
\def\sz{\sigma}  
\def\lz{\lambda} \def\llz{\Lambda}
\def\vez{\varepsilon}  \def\ez{\epsilon}

\def\xz{\xi}     

\def\cM{{\mathcal M}}  \def\cG{{\mathcal G}}
\def\cP{{\mathcal P}}
\def\cC{{\mathcal C}} \def\cD{{\mathcal D}}    \def\cL{{\mathcal L}}
\def\cH{{\mathcal H}}   \def\cB{{\mathcal B}} \def\cS{{\mathcal S}}
\def\cZ{{\mathcal Z}}  
\def\cI{{\mathcal I}}    \def\cQ{{\mathcal Q}}  \def\cR{{\mathcal R}}
\def\cA{{\mathcal A}}    
\def\cK{{\mathcal K}}  \def\cT{{\mathcal T}}
\def\cF{{\mathcal F}}
\def\cW{{\mathcal W}}
\def\tcF{\tilde{\mathcal F}}

\def\tcG{\tilde{\mathcal G}}
\def\tcQ{\tilde{\mathcal Q}}
\def\tphi{\tilde{\phi}}
\def\tez{\tilde{\epsilon}}
\def\bbM{{\mathbb M}}
\def\bbN{{\mathbb N}}  \def\bbZ{{\mathbb Z}}  \def\bbQ{{\mathbb Q}}
    \def\bbF{{\mathbb F}}
\def\bbC{{\mathbb C}}  \def\bbE{{\mathbb E}}  \def\bbP{{\mathbb P}}
    \def\bbD{{\mathbb D}}

         \def\bfU{{\bf U}}
   
  \def\leq{\leqslant}  \def\geq{\geqslant}

\def\lra{\longrightarrow}   
   
\def\ra{\rightarrow}
\def\Hom{\mathrm {Hom}}  
       
\def\Ext{\mathrm{Ext}}   \def\Ker{\mbox{\rm Ker}\,}
\def\dim{\mathrm{dim}}
\def\dime{\mathrm{dim}\,}
\def\End{\mathrm {End}}
\def\Aut{\mbox{\rm Aut}}
\def\IC{\mathrm {IC}}
\def\Id{\mathrm {Id}}
\def\rk{\mathrm {rk}\,}
\def\udim{\mathrm{\underline{dim}}\,}
\def\mod{\mbox{\rm mod}\,}  
\def\rad{\mbox{\rm rad}\,}  
\def\rep{\mbox{\rm Rep}}  %{\mbox{\small\rm sm}}}

\def\sgn{\mbox{\rm sgn}\,}

\def\Ind{\mathrm {Ind}}

\def\fkm{{\mathfrak m}}

\def\GL{\mathrm{GL}}

\def\bfH{{\mathbf H}}  
\def\bfc{{\bf c}}  \def\bfd{{\bf d}} \def\bfe{{\bf e}}

\def\supp{{\rm supp}\,}

\def\Rep{{\rm Rep}}

\def\bfi{{\bf i}}

\def\bfh{{\bf h}}
\def\bfj{{\bf j}}

\def\bfB{{\bf B}}
\def\bfI{{\bf I}}

\def\bfone{{\bf 1}}

\def\bmbz{\bm{\beta}}

\begin{document}

\title[Tame quivers and affine bases II]
{Tame quivers and affine bases II: nonsimply-laced cases}

\author[J. Xiao]{Jie Xiao}
\address{School of mathematical seciences, Beijing Normal University, Beijing 100875, P. R. China}
\email{jxiao@bnu.edu.cn (J. Xiao)}
\author[H. Xu]{Han Xu}
\address{School of Mathematical and Statistics, Beijing Institute of Technology, Beijing 100081, P. R. China}
\email{7520220010@bit.edu.cn (H. Xu)}

\thanks{Jie Xiao was supported by NSF of China (No. 12031007).}

\date{\today}


\bibliographystyle{abbrv}

\maketitle

\begin{abstract}
In \cite{XXZ_Tame_quivers_and_affine_bases_I}, we give a Ringel-Hall algebra approach to the canonical bases in the symmetric affine cases. In this paper, we extend the results to general symmetrizable affine cases by using Ringel-Hall algebras of representations of a valued quiver.
We obtain a bar-invariant basis $\bfB'=\{C(\bfc,t_\lz)|(\bfc,t_\lz)\in\cG^a\}$ in the generic composition algebra $\cC^*$ and prove that $\cB'=\bfB'\sqcup(-\bfB')$ coincides with Lusztig's signed canonical basis $\cB$. Moreover, in type $\tilde{B}_n,\tilde{C}_n$, $\bfB'$ is the canonical basis $\bfB$.
\end{abstract}

\setcounter{tocdepth}{1}\tableofcontents

\section{Introduction}

Given a symmetrizable Cartan datum $(\bfI,(-,-))$ and $\bfU^+$ be the positive part of the associated quantized enveloping algebra over $\bbQ(v)$, generated by $E_\bfi,\bfi\in\bfI$, subject to the quantum Serre relation. Let $\bfU^+_\cA$ be the integral form, where $\cA=\bbZ[v,v^{-1}]$.

By the geometric construction of Lusztig (\cite{Lusztig_Introduction_to_quantum_groups}), there is a unique $\cA$-basis $\bfB$ of $\bfU^+_\cA$ such that $\bfB$ is almost-orthogonal, bar-invariant and compatible with Kashiwara's operators. This basis $\bfB$ is called the canonical basis. We recall this construction here briefly.

Any symmetrizable Cartan datum can be associated to a quiver $Q=(I,H)$ with an automorphism $\sz$ such that $\bfI$ is the set of $\sz$-orbits in $I$.
Given an $I$-graded space $V$ and an automorphism $\sz: V\ra V$ such that $\sz(V_i)=V_{\sz(i)}$, we have the representation variety $\bbE_V$ of $Q$ while $\GL_V$ and acts upon it.    
Lusztig constructed a category $\cQ_V$ containing perverse sheaves on $\bbE_V$ obtained from flag varieties.
The perverse sheaves in $\cQ_V$ are called Lusztig's sheaves.
Then $\bfU^+_{|V|}$ is isomorphic to the Grothendieck group $K(\tcQ_V)$ of $\tcQ_V$, where $\tcQ_V$ is the category of $(L,\phi)$ such that $L\in\cQ_V$ and $\phi:\sz^*L\ra L$ is an isomorphism. Here $\sz:\bbE_V\ra\bbE_V$ is induced naturally by $\sz: V\ra V$. 
The geometric realization is given by the isomorphism 
$$\Xi_2:\bfU^+_\cA\ra K(\tcQ),$$
where $K(\tcQ)=\oplus_{|V|\in\bbN I^\sz} K(\tcQ_V)$.

Let $\cB''$ be the set of all $\pm [L,\phi]$ in $K(\tcQ)$ such that $(D(L),D(\phi)^{-1})\cong (L,\phi)$, where $D$ is the Verdier duality. 
Then the signed canonical basis $\cB$ for $\bfU^+$ is by definition $\Xi_2^{-1}(\cB'')$.
The canonical basis $\bfB$ is by definition a subset of $\cB$ such that $\cB=\bfB\sqcup(-\bfB)$ and $\bfB$ is compatible with Kashiwara's operators. Let $\bfB''=\Xi_2(\bfB)$.

In \cite{Green_Hall_algebras_hereditary_algebras_and_quantum_groups} and \cite{Ringel_Hall_algebras_and_quantum_groups}, the Ringel-Hall algebra realization of $\bfU^+$ is introduced. Any symmetrizable Cartan datum can be associated to a valued quiver $\ggz$ with vertex set $\bfI$. Consider the category $\Rep_k \ggz$ of finite-dimensional representations of $\ggz$ over a finite field $k$. We can define the twisted Ringel-Hall algebra $\cH^*_k(\ggz)$ over $\bbQ$, then define the generic composition algebra $\cC^*_\cA=\cC^*(\ggz)$ over $\bbQ(v)$ as a $\cA$-subalgebra of $\cH^*(\ggz)=\prod_{k\in\cK}\cH^*_k(\ggz)$ for some infinite set $\cK$ of finite fields, where $v$ acts on $\cH^*(\ggz)$ by $(\sqrt{|k|})_k$. Then there is an isomorphism 
$$\Xi_1:\bfU^+_\cA\ra \cC^*_\cA$$
mapping Chevalley generators to simple modules. 
Denote 
$$\Xi=\Xi_2\circ(\Xi_1^{-1}):\cC^*_\cA\ra K(\tcQ).$$

For any symmetric Cartan datum, the construction above can be rebuilt without considering the automorphism $\sz$. In this case, $\Xi_1:\bfU^+_\cA\ra K(\cQ)$ maps the canonical basis to the Grothendieck image of all simple Lusztig's sheaves. 

We consider the affine types in this paper. 
In simply-laced affine cases, the Lusztig' sheaves can be explicitly given. By \cite{Lusztig_Affine_quivers_and_canonical_bases} and \cite{Li_Notes_on_affine_canonical_and_monomial_bases}, if $Q$ is a tame quiver, all simple Lusztig's sheaves are of the form $\IC(\overline{X(\bfc,m)},\cL_\lz)$, where $X(\bfc,m)$ is a locally closed subset of $\bbE_V$ and $\cL_\lz$ is a local system, and simply denoted by $\IC(\bfc,t_\lz)$.

In \cite{XXZ_Tame_quivers_and_affine_bases_I}, Xiao-Xu-Zhao give a Ringel-Hall algebra approach to the canonical basis of simply-laced affine type. 
Let $\cH^*(kQ)$ be the Ringel-Hall algebra of $Q$ over a finite field $k$. 
Then the generic composition subalgebra $\cC^*$ can be constructed and is isomorphic to $\bfU^+$.
Let $\cH^0$ be the extended composition algebra with a PBW-type $\cA$-basis $\{N(\bfc,t_\lz)|(\bfc,t_\lz)\in\cG(Q)\}$ 
$$N(\bfc,t_\lz)=\lan M(\bfc_-)\ran*\lan M(\bfc_0)\ran *S_\lz*\lan M(\bfc_+)\ran,$$
then $\cC^*\subset\cH^0$.
Let $\cG^a$ be the subset of $\cG=\cG(Q)$ containing all aperiodic indices, then Xiao-Xu-Zhao constructed a monomial basis $\{\fkm^{\oz(\bfc,t_\lz)}|(\bfc,t_\lz)\in\cG^a\}$ and a PBW basis $\{E(\bfc,t_\lz)|(\bfc,t_\lz)\in\cG^a\}$ for $\cC^*$ such that
$$\fkm^{\oz(\bfc,t_\lz)}=E(\bfc,t_\lz)+\sum_{(\bfc',t_{\lz'})\in\cG^a\atop(\bfc',t_{\lz'})\prec(\bfc,t_\lz)}\phi_{(\bfc,t_\lz)}^{(\bfc',t_{\lz'})}(v)E(\bfc',t_{\lz'})$$
and
$$E(\bfc,t_\lz)=N(\bfc,t_\lz)+\sum_{(\bfc',t_{\lz'})\in\cG\setminus\cG^a\atop (\bfc',t_{\lz'})\prec(\bfc,t_\lz)}b_{(\bfc,t_\lz)}^{(\bfc',t_{\lz'})}(v)N(\bfc',t_{\lz'})$$
hold in $\cH^0$, where coefficients are in $\cA$. 
Then there exists a bar-invariant basis $\bfB'=\{C(\bfc,t_\lz)|(\bfc,t_\lz)\in\cG^a\}$ of $\cC^*$ such that
$$C(\bfc,t_\lz)=E(\bfc,t_\lz)+\sum_{(\bfc',t_{\lz'})\in\cG^a,(\bfc',t_{\lz'})\prec(\bfc,t_\lz)}g^{(\bfc',t_{\lz'})}_{(\bfc,t_\lz)}(v)E(\bfc',t_{\lz'})$$
for some $g^{(\bfc',t_{\lz'})}_{(\bfc,t_\lz)}(v)\in v^{-1}\bbZ[v^{-1}]$.
Let $B(\bfc,t_\lz)=[\IC(\bfc,t_\lz)]\in\bfB''$ be the canonical basis element in Lusztig's geometric realization, then we have 
\begin{theorem}
    For simply-laced affine type, we have $\Xi(\bfB')=\bfB''$. Moreover, we have $\Xi(C(\bfc,t_\lz))=B(\bfc,t_\lz)$.
\end{theorem}

In this paper, we aim to generalize this result to all affine type. First, we construct the extended composition algebra $\cH^0$ for a tame valued quiver, and prove that it is closed under multiplication and comultiplication. Then the monomial basis $\{\fkm^{\oz(\bfc,t_\lz)}|(\bfc,t_\lz)\in\cG^a\}$, PBW basis $\{E(\bfc,t_\lz)|(\bfc,t_\lz)\in\cG^a\}$ and a bar-invariant basis $\bfB'=\{C(\bfc,t_\lz)|(\bfc,t_\lz)\in\cG^a\}$ can be constructed similarly to simply-laced cases. 
Let $\cB'=\bfB'\sqcup(-\bfB')$. By comparing to the geometric signed canonical basis $\cB''=\bfB''\sqcup(-\bfB'')$, where
$$\bfB''=\{B(\bfc,t_\lz)=[\IC(\eta(\bfc),t_\lz),\tphi]|(\bfc,t_\lz)\in\cG^a\},$$
here $\tphi$ is uniquely chosen for each Lusztig's sheaf,
we can prove the main theorem:
\begin{theorem}
For general affine type, we have $\Xi(\cB')=\cB''$.  Moreover, we have $\Xi(\pm C(\bfc,t_\lz))=\pm B(\bfc,t_\lz)$.
\end{theorem}
Moreover, for type $\tilde{B}_n$ and $\tilde{C}_n$, by using Kashiwara's operators, we prove that $\Xi(C(\bfc,t_\lz))=B(\bfc,t_\lz)$, thus $\Xi(\bfB')=\bfB''$.

\section{Preliminary}

\subsection{Valued quivers}

A quiver $Q=(I,H,s,t)$ consists of the following data: a set $I$ of vertices, a set $H$ of arrows, and two maps $s,t:H\ra I$ such that any arrow $h\in H$ starts at $s(h)$ and terminates at $t(h)$. 

A valued quiver $\ggz=(\bfI,\bfH,s,t;d,m)$ consists of the following data: a quiver $(\bfI,\bfH,s,t)$ and two maps $d:\bfI\ra\bbZ^+,\bfi\mapsto d_\bfi;m:\bfH\ra\bbZ^+,\bfh\mapsto m_\bfh$ such that for any arrow $\bfh\in\bfH$, $m_{\bfh}$ is a common multiple of $d_{s(\bfh)}$ and $d_{t(\bfh)}$. We say that $d,m$ are the valuations of $\ggz$ on vertices and arrows, respectively. A quiver can be treated as a valued quiver given that the valuations are trivial ($d_{(-)}=1=m_{(-)}$).

We require in this paper that all quivers and valued quivers are finite and have no loops, that is, $I,H,\bfI,\bfH$ are finite sets there is no arrow sharing the same starting and terminal vertex. All quivers and valued quivers shall admit no oriented cycle (a sequence of arrows $h_1,h_2,\dots,h_s$ such that $t(h_r)=s(h_{r+1})$ for $1\leq r\leq s-1$ and $t(h_s)=s(h_1)$) unless stated.

An automorphism $\sz$ of a quiver $Q=(I,H,s,t)$ is a pair of bijections $I\ra I$ and $H\ra H$, both denoted by $\sz$, such that for any $h\in H$, $\sz(s(h))=s(\sz(h))$, $\sz(t(h))=t(\sz(h))$. For example, the identity  $\Id_Q$ consisting of identity maps of $I$ and $H$ is an automorphism.

Given a quiver $Q=(I,H,s,t)$ and an automorphism $\sz$ of $Q$ as above, we can define a valued quiver $\ggz(Q,\sz)=(\bfI,\bfH,s,t;d,m)$ as follows: the sets $\bfI$ and $\bfH$ are respectively the sets of $\sz$-orbits in $I$ and $H$; the maps $s,t$ are naturally induced; the cardinal number of a $\sz$-orbit $\bfi\in\bfI$ equals to $d_\bfi$; the cardinal number of a $\sz$-orbit $\bfh\in\bfH$ equals to $m_\bfh$.

\subsection{Representations}

Let $\ggz=(\bfI,\bfH,s,t;d,m)$ be a valued quiver and $k$ be a field. A $k$-modulation (also called $k$-species in \cite{Dlab_Ringel_Indecomposable_representations_of_graphs_and_algebras}) $\bbM=(\bbD_\bfi,\bbM_\bfh)_{\bfi\in\bfI,\bfh\in\bfH}$ on $\ggz$ is a family of divisible algebras $\bbD_\bfi,\bfi\in\bfI$ over $k$ and  $\bbD_{t(\bfh)}-\bbD_{s(\bfh)}-$bimodules $\bbM_\bfh,\bfh\in\bfH$, such that $\dim_k \bbD_\bfi=d_\bfi,\dim_k \bbM_\bfh=m_\bfh$, and the actions of  $k$ on $\bbM_\bfh$ from the left side and the right side are compatible. We call the pair $(\ggz,\bbM)$ a $k$-modulated (valued) quiver.

For a given $k$-modulated quiver $(\ggz,\bbM)$ as above, a representation (or module)  $V= (V_\bfi,V_\bfh)_{\bfi\in\bfI,\bfh\in\bfH}$ of $(\ggz,\bbM)$ consists of a family of $\bbD_{\bfi}$-vector spaces $V_\bfi$, $\bfi\in\bfI$ and a family of $\bbD_{t(\bfh)}$-linear maps
$$V_\bfh: \bbM_\bfh\otimes_{\bbD_{s(\bfh)}}V_{s(\bfh)}\lra V_{t(\bfh)},\bfh\in\bfH.$$

Given two representations $V$ and $V'$ of $(\ggz,\bbM)$, a morphism $f$ from $V$ to $V'$ is a family of $\bbD_\bfi$-linear maps $f_\bfi: V_\bfi\ra V'_\bfi$, $\bfi\in\bfI$ such that the following diagram commutes for all $\bfh\in\bfH$:
$$\xymatrix{
\bbM_\bfh\otimes_{\bbD_{s(\bfh)}}V_{s(\bfh)} \ar[r]^{V_\bfh} \ar[d]^{\Id_{\bbM_\bfh}\otimes f_{s(\bfh)}} &   V_{t(\bfh)} \ar[d]^{f_{t(\bfh)}} \\
\bbM_\bfh\otimes_{\bbD_{s(\bfh)}}V'_{s(\bfh)} \ar[r]^{V'_\bfh}   & V'_{t(\bfh)}
}$$

All representations of $(\ggz,\bbM)$ and their morphisms form a $k$-category, denoted by $\rep_k (\ggz,\bbM)$. We often abbreviate the $k$-modulation $\bbM$ and denote it by $\rep_k \ggz$. For $V\in\rep_k \ggz$, we denote by $[V]$ the isomorphism class of $V$ and by $\udim V=\sum_{\bfi\in\bfI}(\dim_k V_\bfi)\bfi\in\bbZ \bfI$ the dimension vector of $V$. We denote by $S_\bfi$ the simple module at $\bfi$.

For the special case where $\ggz$ is a quiver (valuations are trivial), let the $k$-modulation be such that all divisible algebras and bimodules are $k$, then the definition of $\rep_k \ggz$ here is exactly the definition of quiver representations as commonly known.

\subsection{Cartan datum}\label{sec: Cartan datum}

A Cartan datum is a pair $(I,(-,-))$ consisting of a finite set $I$ and a symmetric bilinear form $\nu,\nu'\mapsto (\nu,\nu')$ on the free abelian group  $\bbZ I$ with values in $\bbZ$ such that:
\begin{enumerate}
\item[(a)] $(i,i)\in\{2,4,6,\dots\}$ for any $i\in I$;
\item[(b)] $2\frac{(i,j)}{(i,i)}\in\{0,-1,-2,\dots\}$ for any $i\neq j$ in $I$.
\end{enumerate}

There is also an equivalent definition. A Cartan datum is a triple $(I,C,D)$ consisting of a finite set $I$ and two $I\times I$-matrices $C,D$ in $\bbZ$ such that:
\begin{enumerate}
\item[(a)] $C_{ii}=2$ for any $i\in I$; $C_{ij}\in\{0,-1,-2,\dots\}$ for any  $i\neq j$ in $I$;
\item[(b)] $D$ is a diagonal matrix with positive integers in the diagonal and $DC$ is symmetric.
\end{enumerate}
The matrix $C$ is often called a symmetrizable generalized Cartan matrix.

The equivalence is given as follow. For a given $(I,(-,-))$, one can obtain $(I,C,D)$ by letting $C_{ij}=2\frac{(i,j)}{(i,i)}$ and $D_{ii}=\frac{(i,i)}{2}$ for all $i,j\in I$. On the other hand, one can obtain $(I,(-,-))$ from $(I,C,D)$ by setting $(i,j)$ to be the $i,j$-entry of $DC$ for all $i,j\in I$.

Given a valued quiver $\ggz=(\bfI,\bfH,s,t;d,m)$, we associate it with a generalized Cartan matrix $C=C_\ggz=(C_{\bfi\bfj})_{\bfi,\bfj\in\bfI}$ defined as:
$$C_{\bfi\bfj}=\left\{
\begin{aligned}
&2,&\text{if}\, \bfi=\bfj,\\
&-\sum_\bfh m_\bfh/ d_\bfi,&\text{if}\, \bfi\neq\bfj.
\end{aligned}
\right.
$$
Here the sum $\sum_\bfh$ is taken over all arrows between $\bfi$ and $\bfj$.

Let $D=D_\ggz=\text{diag}\{d_\bfi\}_{\bfi\in\bfI}$, then one can check that $(\bfI,C,D)$ is a Cartan datum. We call it the Cartan datum of $\ggz$. Need to remind that $(\bfI,C,D)$ is independent of the orientation of $\ggz$, that is, if we exchange the starting and terminal vertices of an arrow in $\ggz$, we will get the same Cartan datum.

\subsection{Root system and Weyl group}

For a  Cartan datum $(I,(-,-))$, we can define a simple reflection $s_i\in\Aut(\bbZ I)$ for all $i\in I$:
$$s_i(\nu)=\nu-\frac{2(\nu,i)}{(i,i)}i.$$
The Weyl group $\cW$ is by definition the subgroup of $\Aut(\bbZ I)$ generated by all $s_i$.

Any $w\in\cW$ can be written as an expression $w=s_{i_1}s_{i_2}\dots s_{i_l}$. We say this expression is reduced if $l$ is minimal among all possible expressions of $w$.

Denote by $\Phi_{\rm re}=\cW I$ the set of real roots. Then $\Phi_{\rm re}=\Phi_{\rm re}^+\sqcup\Phi_{\rm re}^-$ where $\Phi_{\rm re}^+=\Phi_{\rm re}\cap\bbN I$ and $\Phi_{\rm re}^-=-\Phi_{\rm re}^+$. Denote by $\Phi_{\rm im}=\rad (-,-)$ the set of imaginary roots.

\subsection{Euler form}

For a given valued quiver $\ggz$ and $\nu,\nu'\in\bbZ \bfI$, we define the Euler form 
$$\lan \nu,\nu'\ran=\sum_{\bfi\in\bfI}d_\bfi\nu_\bfi\nu'_\bfi-\sum_{\bfh\in\bfH}m_\bfh\nu_{s(\bfh)}\nu'_{t(\bfh)}$$
and the symmetric Euler form 
$$(\nu,\nu')=\lan \nu,\nu'\ran+\lan \nu',\nu\ran,$$
where $\nu=\sum_{\bfi\in\bfI}\nu_\bfi \bfi,\nu'=\sum_{\bfi\in\bfI}\nu'_\bfi \bfi$.

One can easily check that $(\bfI,(-,-))$ is a Cartan datum and equivalent to $(\bfI,C,D)$ defined in Section $\ref{sec: Cartan datum}$. Note that $\lan -,-\ran$ depends on the orientation of $\ggz$ while $(-,-)$ does not.

By \cite{Ringel_Representations_of_K-species_and_bimodules}, given $V,W\in\rep_k \ggz$, it is known that
$$\lan \udim V,\udim W\ran=\dim_k \Hom_\ggz(V,W)-\dim_k \Ext^1_\ggz(V,W).$$

\subsection{Quantum groups}

The quantum group $\bfU$ (also called the quantized enveloping algebra) associated to a Cartan datum $(I,(-,-))$ or $(I,C,D)$ as above is a $\bbQ(v)$-algebra with generators $E_i,F_i,i\in I$ and $K_\nu,\nu\in\bbZ I$ subject to the following relations:
\begin{enumerate}
\item[(1)] $K_0=1$, $K_\nu K_{\nu'}=K_{\nu+\nu'}$, $\nu,\nu'\in\bbZ I$;
\item[(2)] $K_\nu E_i=v^{(\nu,i)}E_iK_\nu$, $K_\nu F_i=v^{(-\nu,i)}F_iK_\nu$, $\nu\in\bbZ I,i\in\ I$;
\item[(3)] $E_iF_j-F_jE_i=\dz_{ij}\frac{\tilde{K}_i-\tilde{K}_{-i}}{v_i-v_i^{-1}}$, $i,j\in I$;
\item[(4)](Quantum Serre relations) $\sum_{p+p'=1-C_{ij}}(-1)^pE_i^{(p)}E_jE_i^{(p')}=0$, $i\neq j$.
\end{enumerate}
Here $v_i=v^{d_i},[p]_{x}=\frac{x^p-x^{-p}}{x-x^{-1}},[p]^!_{x}=[p]_x[p-1]_x\dots [1]_x, [0]_x=1, E_i^{(p)}=E_i^p/[p]^!_{v_i}.\tilde{K}_{\pm i}=K_{\pm ((i,i)/2) i}$.

Let $\bfU^+,\bfU^0,\bfU^-$ be the $\bbQ(v)$-subalgebras of $\bfU$ generated by $E_i,i\in I$; $K_\nu,\nu\in\bbZ I$; $F_i,i\in I$, respectively. It is well-known that there is a triangular decomposition as vector spaces: $\bfU=\bfU^+\otimes\bfU^0\otimes\bfU^-$.

Let $\cA=\bbZ[v,v^{-1}]$. The integral form $\bfU_\cA^+$ is by definition the $\cA$-subalgebra of $\bfU$ generated by all $E_i^{(m)},i\in I,m\geq 0$.

The bar involution $\bar{}:\bfU^+\ra\bfU^+$ is the $\bbQ$-algebra isomorphism such that $\bar{E_\bfi}=E_\bfi$ and $\bar{v}=v^{-1}$.

\subsection{Ringel-Hall algebras}\label{sec: RH alg}

Let $k$ be a finite field $v_k=\sqrt{|k|}$. The (twisted) Ringel-Hall algebra $\cH^*_k(\ggz)$ of $\rep_k \ggz$ is by definition a $\bbQ(v_k)$-algebra with basis $\{[M]|M\in\rep_k \ggz\}$ and multiplication $*$ defined as
$$[M]*[N]=\sum_{[L]}v_k^{\lan \udim M,\udim N\ran}g^{L}_{MN}[L],$$
where $g^L_{MN}=|\{V\subset_{\text{submodule}} L| V\cong N, L/V\cong M\}|$ is the Hall number.

Following \cite{Ringel_PBW-bases_of_quantum_groups}, for any $M\in\rep_k\ggz$, we denote
$$\lan M\ran =v_k^{-\dim_k M+\dim_k \End_\ggz M}[M]$$

Let $\cK$ be an infinite set of finite fields $k$ such that for any $n\in\bbN$, there exists $k\in\cK$ with $|k|\geq n$. 
Consider the direct product
$$\cH^*(\ggz)=\prod_{k\in\cK}\cH^*_k(\ggz)$$
and $v,v^{-1},u_\bfi$ are elements of $\cH^*$ whose $k$-components are $v_k,v_k^{-1},[S_\bfi]$, respectively.
The generic composition algebra $\cC_\cA^*(\ggz)$ is defined as the subalgebra of $\cH^*$ generated by $\cA$ and $u_\bfi^{(m)},\bfi\in\bfI,m\geq 0$. Define the $\bbQ(v)$-algebra
$$\cC^*(\ggz)=\bbQ(v)\otimes_{\cA}\cC_\cA^*(\ggz).$$

\begin{lemma}[\cite{Green_Hall_algebras_hereditary_algebras_and_quantum_groups}]
Let $(\bfI,(-,-))$ be the Cartan datum of $\ggz$ and $\bfU$ be the associated quantum group. Then, by abuse of notation, there is a $\bbQ(v)$-algebra isomorphism $\Xi_1:\bfU^+\ra\cC^*(\ggz)$ and an $\cA$-algebra isomorphism $\Xi_1:\bfU_\cA^+\ra\cC_\cA^*(\ggz)$, both given by $E_\bfi^{(m)}\mapsto u_\bfi^{(m)}$.
\end{lemma}

\section{About $\rep_k\ggz$}

\subsection{Affine and tame}

A Cartan datum $(I,(-,-))$ or $(I,C,D)$ is of finite (resp., affine) type if $DC$ is positive definite (resp., positive semidefinite of corank 1). 

A quiver $Q$ is tame if the underlying graph of $Q$ is of type $\tilde{A}_n$, $\tilde{D}_n$, $\tilde{E}_6$, $\tilde{E}_7$, $\tilde{E}_8$. A valued quiver $\ggz$  is tame if it is isomorphic to $\ggz(Q,\sz)$ for a tame quiver $Q$ and an automorphism $\sz$ of $Q$. 
The Cartan datum of a valued quiver $\ggz$ is of affine type if and only if $\ggz$ is tame.

We use notations in \cite{Dlab_Ringel_Indecomposable_representations_of_graphs_and_algebras} for the classification of all tame quivers and tame valued quivers. 
The types are listed as follows: $\tilde{A}_{11}$, $\tilde{A}_{12}$, $\tilde{A}_n$, $\tilde{B}_n$, 
$\tilde{C}_n$, $\Tilde{BC}_n$, $\tilde{BD}_n$, $\tilde{CD}_n$, $\tilde{D}_n$, $\tilde{E}_6$, $\tilde{E}_7$, $\tilde{E}_8$, $\tilde{F}_{41}$, $\tilde{F}_{42}$, $\tilde{G}_{21}$, $\tilde{G}_{22}$. 
Here types $\tilde{A}_n$, $\tilde{D}_n$, $\tilde{E}_6$, $\tilde{E}_7$, $\tilde{E}_8$ are known as simply-laced.

\subsection{Frobenius folding}\label{sec Frob fold}

To describe $\rep_k\ggz$ via $(Q,\sz)$, we introduce the notion of Frobenius folding and refer to \cite{Deng_Du_Frobenius_morphisms_and_representations_of_algebras} for more details.

Let $k$ be a finite field and $\bar{k}$ be an algebraic closure of $k$. Let $Q$ be a quiver (not necessarily of finite or affine type) with an automorphism $\sz$. Consider the category $\rep_{\bar{k}} Q$ of representations of $Q$ over $\bar{k}$. By \cite{DDPW_Finite_dimensional_algebras_and_quantum_groups}, there is a self-equivalent functor $(-)^{[1]}$ on $\rep_{\bar{k}} Q$ such that $\udim M^{[1]}=\sz(\udim M)$. The functor $(-)^{[1]}$ is called the Frobenius twist functor and denote $(-)^{[n]}=((-)^{[1]})^n$.

For $M$ in $\rep_{\bar{k}} Q$, let $r$ be the minimal positive integer such that $M\cong M^{[r]}$. If $r=1$, we say $M$ is $(-)^{[1]}$-stable. By \cite{DDPW_Finite_dimensional_algebras_and_quantum_groups}, there is a Frobenius map $F$ on $M$ such that the fixed point set $M^{F}$ is in $\rep_k\ggz$ and all representations in $\rep_k\ggz$ can be obtained in this way (up to isomorphisms).

Such $M^{F}$ is in $\rep_k\ggz$ for a certain $k$-modulation. Since the isomorphism between $\bfU^+$ and $\cC^*$ does not depend on which $k$-modulation we choose, we abbreviate it.
Also, the isomorphism class of $M^{F}$ does not depend on the choice of the Frobenius map $F$.

Moreover, this gives a bijection 
$$\{\text{iso-classes of }\rep_k\ggz\}\leftrightarrow\{ \text{iso-classes of }(-)^{[1]}-\text{stable objects in }\rep_{\bar{k}} Q\}$$
and a bijection
\begin{eqnarray}
&\{\text{iso-classes of indecomposable modules in }\rep_k\ggz\}\leftrightarrow\notag \\
&\{(-)^{[1]}-\text{orbits of iso-classes of indecomposable modules in }\rep_{\bar{k}} Q\}.\notag
\end{eqnarray}

If a full subcategory $\cC$ of $\rep_{\bar{k}} Q$ such that $N\in\cC$ implies $N^{[1]}\in\cC$, then there is a corresponding full subcategory $\cC^F$ of $\rep_k \ggz$ consisting of the objects which are isomorphic to $M^F$ for some $(-)^{[1]}$-stable $M\in\cC$ and a Frobenius map $F$. We say $\cC^F$ is the Frobenius folding of $\cC$. In particular, $\rep_k \ggz$ is the Frobenius folding of $\rep_{\bar{k}} Q$.

\subsection{Representations of a valued tame quiver} 

For a tame quiver $Q$, denote by $\cP,\cR,\cI$ the full subcatgory of $\rep_{\bar{k}} Q$ consisting of preprojective, regular , preinjective modules, respectively. Denote by $\cR_{\rm nh}$ and $\cR_{\rm h}$ the full subcategory of $\cR$ consisting of nonhomogeneous and homogeneous regular modules. 

Let $K_r$ be the cyclic quiver of $r$-vertices and $\cK_r(f)=\rep^0_{f}K_r$ be the category of nilpotent representations of $K_r$ for any field $f$ and $r\geq 1$. We say a category is a tube of rank $r$ over a field $f$ if it is equivalent to $\cK_r(f)$. We say a tube is nonhomogeneous if rank $\geq 2$ and homogeneous if rank is 1.

Then $\cR_{\rm nh}$ is a direct product of finitely many subcategories which are nonhomogeneous tubes over $\bar{k}$ and $\cR_{\rm h}$ is a direct product of infinitely many subcategories which are homogeneous tubes over $\bar{k}$.

For a valued tame quiver $\ggz=\ggz(Q,\sz)$, let $\cP^\ggz,\cR^\ggz,\cR^\ggz_{\rm nh},\cR^\ggz_{\rm h},\cI^\ggz$ be the full subcategory of $\rep_k\ggz$ consisting of preprojective, regular, nonhomogeneous regular, homogeneous regular, preinjective modules, respectively. Let $\cP^F,\cR^F,\cR^F_{\rm nh},\cR^F_{\rm h},\cI^F$ be the Frobenius foldings of $\cP,\cR,\cR_{\rm nh},\cR_{\rm h},\cI$.  

We have 
$$\cP^\ggz=\cP^F,\cR^\ggz=\cR^F,\cI^\ggz=\cI^F,$$
and
$$\cR^\ggz_{\rm nh}\subset\cR^F_{\rm nh}, \cR^F_{\rm h}\subset\cR^\ggz_{\rm h}$$
since a nonhomogeneous tube could be Frobenius-folded to a homogeneous tube for some $Q$ and $\sz$.

\subsection{Sequence $\underline{\bfi}$}

A doubly infinite sequence 
$$\underline{\bfi}=(\dots,\bfi_{-1},\bfi_0|\bfi_1,\bfi_2,\dots)$$
in $\bfI$ is called reduced if $s_{\bfi_{r}}s_{\bfi_{r+1}}\dots s_{\bfi_{t-1}}s_{\bfi_{t}}$ is a reduced expression for all $r,t\in\bbZ$ with $r<t$.

A vertex $\bfi\in\bfI$ is called a sink (resp., source) if there is no arrow starting (resp., terminating) at $\bfi$. Let $\sz_\bfi\ggz$ be the valued quiver obtained from $\ggz$ by reversing all arrows connected to $\bfi$ and replacing $\bbM_\bfh$ by its $k$-dual for $s(\bfh)=\bfi$ or $t(\bfh)=\bfi$.
We say that a finite sequence $(\bfi_1,\dots,\bfi_m)$ is an admissible sink (resp., source) sequence of $\ggz$ if for any $1\leq t \leq m$, $\bfi_t$ is a sink (resp., source) of $\sz_{\bfi_{t-1}}\dots\sz_{\bfi_1}\ggz$. 

We say that $\underline{\bfi}$ is admissible to $\ggz$ if $\underline{\bfi}$ is reduced and $(\bfi_0,\bfi_{-1},\dots,\bfi_{-m})$ is admissible sink sequence of $\ggz$ and $(\bfi_1,\bfi_{2},\dots,\bfi_{m})$ is admissible source sequence of $\ggz$ for all $m> 0$.

\subsection{Preprojective and preinjective part}\label{sec: PI part}

For a given doubly infinite sequence $\underline{\bfi}$ which is admissible to $\ggz$, set positive real roots
$$\bmbz_t=\left\{
\begin{aligned}
&s_{\bfi_0}s_{\bfi_{-1}}\dots s_{\bfi_{t+1}}(\bfi_t),   & \text{if}\, t\leq 0,\\
&s_{\bfi_1}s_{\bfi_{2}}\dots s_{\bfi_{t-1}}(\bfi_t),    & \text{if}\, t>0.
\end{aligned}
\right.
$$
Here $\bmbz_0=\bfi_0,\bmbz_1=\bfi_1$. Then there is a unique (up to isomorphism) indecomposable module $M_k(\bmbz_t)$ in $\rep_k\ggz$ for each $t\in\bbZ$ such that 
$$\udim M_k(\bmbz_t)=\bmbz_t.$$
Moreover, we can classify all indecomposable module in $\cP^\ggz$ and $\cI^\ggz$:
$$\{\text{iso-classes of indecomposable modules in }\cP^\ggz\}=\{[M_k(\bmbz_t)]|t\leq 0\},$$
$$\{\text{iso-classes of indecomposable modules in }\cI^\ggz\}=\{[M_k(\bmbz_t)]|t>0\}.$$
Note that $\bmbz_t,t\in\bbZ$ are not all of the positive real roots if $\cR^\ggz_{\rm nh}$ is nontrivial.
\begin{lemma}
For $t_1<t_2\leq 0$ or $0<t_1<t_2$, we have 
$$\Hom_\ggz(M_k(\bmbz_{t_1}),M_k(\bmbz_{t_2}))=0$$
$$\Ext^1_\ggz(M_k(\bmbz_{t_2}),M_k(\bmbz_{t_1}))=0$$
\end{lemma}
We order these real roots:
$$(\bmbz_0<\bmbz_{-1}<\cdots)<(\cdots<\bmbz_2<\bmbz_1).$$

Let $\cG_-$ (resp., $\cG_+$) be the set of all finitely supported functions $\bbZ_{\leq 0}\ra\bbN$ (resp., $\bbZ_{>0}\ra\bbN$). For $\bfc_-\in\cG_-$ and $\bfc_+\in\cG_+$, let
$$M_k(\bfc_-)=M_k(\bmbz_0)^{\oplus \bfc_-(0)}\oplus M_k(\bmbz_{-1})^{\oplus \bfc_-(-1)}\oplus\dots,$$
$$M_k(\bfc_+)=\dots\oplus M_k(\bmbz_2)^{\oplus \bfc_+(2)}\oplus M_k(\bmbz_{1})^{\oplus \bfc_+(1)}.$$
Then we have the following classification:
$$\{\text{iso-classes of  modules in }\cP^\ggz\}=\{[M_k(\bfc_-)]|\bfc_-\in\cG_-\},$$
$$\{\text{iso-classes of  modules in }\cI^\ggz\}=\{[M_k(\bfc_+)]|\bfc_+\in\cG_+\}.$$

\subsection{Nonhomogeneous regular part}\label{sec: nh part}

A module in a nonhomogeneous tube $\cT$ is called regular simple if it is simple as an object in $\cT$.
To charaterize a nonhomogeneous tube $\cT$ is deduced to charaterize all regular simples in $\cT$.

We refer to Section 6 in \cite{Dlab_Ringel_Indecomposable_representations_of_graphs_and_algebras} where regular simples in $\cR^\ggz_{\rm nh}$ are given for all types of tame valued quiver with a certain orientation. By applying BGP reflection functors (which are equivalences between $\cR^\ggz_{\rm nh}$'s of different orientations), one can obtain all regular simple for other orientations. 

By comparing $\cR^\ggz_{\rm nh}$  with $\cR_{\rm nh}$ for all types of $\ggz=\ggz(Q,\sz)$, we have the following lemma.
\begin{lemma}\label{how nh tubes are F-folded}
    For a nonhomogeneous regular tube $\cT$ in $\cR^\ggz_{\rm nh}$, a regular simple $L$ in $\cT$ can be obtained in the follow manners:
    \begin{enumerate}
        \item[(a)] $L\cong (L_0\oplus L_1 \oplus\dots\oplus L_{l-1})^F$, where $L_t\in\ T^{[t]}$ are regular simples, and $T,T^{[1]},\dots,T^{[l-1]}$ ($l\leq 3$) are different tubes in $\cR_{\rm nh}$ such that $\rk T=\rk\cT$ and $(-)^{[l]}|_T$ is naturally isomorphic to identity on $T$.
        \item[(b)] $L\cong (L_1\oplus L_2)^F$, where $L_1,L_2$ are regular simples in the same tube $T$ of $\cR_{\rm nh}$, such that $\rk T=2\rk \cT\geq 4$. Here $T$ is a non-homogeneous tube in $\Rep_{\bar{k}}Q$.
    \end{enumerate}
\end{lemma}

For example, $\ggz$ of type $\tilde{F}_{41}$ has a tube of rank 3, which is obtained by Frobenius-folding the two rank-3 tubes of $Q$ of type $\tilde{E}_6$, satisfying case $(a)$. 
For another example, let $n>3$, $\ggz$ of type $\tilde{BD}_n$ has a tube of rank $n-1$ which is obtained by Frobenius-folding the tube of rank $2n-2$ of $\tilde{D}_{2n}$, satisfying case $(b)$.

Since $\cR^\ggz_{\rm nh}$ is the Frobenius folding of some nonhomogeneous tubes over $\bar{k}$ in $\cR_{\rm nh}$, $\cR^\ggz_{\rm nh}$ is a direct product of some nonhomogeneous tubes.
Denote by $\cT_1,\cT_2,\dots,\cT_s(s\leq 3)$ the nonhomogeneous tubes in $\cR^\ggz_{\rm nh}$ and fix the equivalences $\vez_t:\cK_{r_t}(k_t)\cong\cT_t$ where $r_t\geq 2$ is the rank of $\cT_t$ and $k_t$ is the endomorphism ring of a simple object in $\cT_t$. The equivalence $\vez_t$ also induces an Hall-algebra isomorphism $\cH^*(\cK_{r_t}(k_t))\cong\cH^*(\cT_t)$, which is denoted again by $\vez_t$.

By \cite{Deng_Du_Xiao_Generic_extensions_and_canonical_bases_for_cyclic_quivers}, the set of isomorphism classes of modules of $\cK_r(k')$ is one-to-one-correspondent to the set $\Pi_r$ of (cyclic) multisegments. The set $\Pi_r$ is independent of $k'$ and its elements are of the form
$$\pi=\sum_{1\leq i\leq r,l\geq 1}\pi_{i,l}[i;l)$$
with $\pi_{i,l}\in\bbN$ and almost all of them being zero. We denote $M_{k'}(\pi)\in\cK_r(k')$ the corresponding module.

Let 
$$\cG_0=\Pi_{r_1}\times\Pi_{r_2}\times\dots\times\Pi_{r_s},$$
then for $\bfc_0=(\pi_1,\pi_2,\dots,\pi_s)\in\cG_0$, we denote 
$$M_k(\bfc_0)=\vez_1(M_{k_1}(\pi_1))\oplus\vez_2(M_{k_2}(\pi_2))\oplus\dots\oplus\vez_s(M_{k_s}(\pi_s))\in\cR^\ggz_{\rm nh},$$
then 
$$\{\text{iso-classes of  modules in }\cR^\ggz_{\rm nh}\}=\{[M_k(\bfc_0)]|\bfc_0\in\cG_0\}.$$
Note that $\cG_0$ is independent of the field $k$ and the choice of the equivalences $\vez_t$'s. 

A multisegment $\pi\in\Pi_r$ is called aperiodic if, for each $l\geq 1$, there is some $i$ such that $\pi_{i,l}=0$. We denote by $\Pi_r^a$ the set of aperiodic multisegments. Denote
$$\cG^a_0=\Pi_{r_1}^a\times\Pi_{r_2}^a\times\dots\times\Pi_{r_s}^a.$$
A module $M$ in $\cR^\ggz_{\rm nh}$ is called aperiodic if $M$ is isomorphic to $M_k(\bfc_0)$ for some $\bfc_0\in\cG_0^a$.

\subsection{Homogeneous regular part}

The homogeneous regular part $\cR^\ggz_{\rm h}$ is a direct product of some homogeneous tubes and we denote by $\cZ=\cZ_k$ the index set of these tubes. 
For $z\in\cZ$, let $\cT_z$ be the corresponding homogeneous tube, then $\cT_z\cong\cK_1(k_z)$ for some extension field $k_z$ of $k$. Let $d_z$ be minimal positive integer such that $d_z\dz$ is a dimension vector of some module in $\cT_z$. We denote $\deg z=d_z$. Note that the set of isomorphism classes of objects in $\cK_1(f)$ is isomorphic to the set of all partitions for any field $f$.

Let $\lz$ be a partition and $M_k(\lz,z)$ be the corresponding module in $\cT_z$. Then any module in $\cR^\ggz_{\rm h}$ is isomorphic to 
$$M_{k}(\lz_1,z_1)\oplus M_{k}(\lz_2,z_2)\oplus \dots\oplus M_{k}(\lz_t,z_t)$$
for some partitions $\lz_1,\dots,\lz_t$ and some pairwise-different $z_1,\dots,z_t\in\cZ$. Denote this module by $M_{k}(\underline{\lz},\underline{z})$, where $\underline{\lz}=(\lz_1,\dots,\lz_t)$, $\underline{z}=(z_1,\dots,z_t)$.

\subsection{Index set}\label{index set}

Following \cite{XXZ_Tame_quivers_and_affine_bases_I}, we define an similar index set for the valued tame quiver $\ggz$. Denote $\cG_-$ (resp., $\cG_+$) the set of all finitely supported functions $\bbZ_{\leq 0}\ra\bbN$ (resp., $\bbZ_{>0}\ra\bbN$). Denote $\bbP$ the set of all $t_\lz$ and $0$, where $t_\lz$ is the complex character of the Specht module $S_\lz$ of a symmetric group over $\bbC$ and $\lz$ is a partition of an positive integer.

Set 
$$\cG'=\cG_-\times\cG_0\times\cG_+,$$
$$\cG'^a=\cG_-\times\cG_0^a\times\cG_+,$$
and 
$$\cG=\cG'\times\bbP,\tcG=\cG'\times\bbN,$$
$$\cG^a=\cG'^a\times\bbP,\tcG^a=\cG'^a\times\bbN.$$

For $\bfc=(\bfc_-,\bfc_0,\bfc_+)\in\cG'$, let 
$$M_k(\bfc)=M_k(\bfc_-)\oplus M_k(\bfc_0)\oplus M_k(\bfc_+).$$
By Section \ref{sec: PI part} and \ref{sec: nh part}, we have
$$\{\text{iso-classes of  modules in }\Rep_k\ggz \text{ without direct summands in }\cR^\ggz_{\rm h}\}= \{[M_k(\bfc)]|\bfc\in\cG'\}$$
and $\cG'$ depends only on $\ggz$.

For $\bfc\in\cG',t_\lz\in\bbP,m\in\bbN$, set 
$$D(\bfc,t_\lz)=\udim M_k(\bfc)+|\lz|\dz\in\bbN \bfI,$$
$$D(\bfc,m)=\udim M_k(\bfc)+m\dz\in\bbN \bfI,$$
where $\dz$ is the minimal positive imaginary root in $\bbN\bfI$, which is also the minimal dimension vector of nonzero modules in $\cR^\ggz_{\rm h}$.
For $\nu\in\bbN\bfI$, denote
$$\cG_\nu=\{(\bfc,t_\lz)\in\cG|D(\bfc,t_\lz)=\nu\},$$
$$\tcG_\nu=\{(\bfc,t_m)\in\tcG|D(\bfc,m)=\nu\}.$$
Then $\cG=\sqcup_{\nu\in\bbN\bfI}\cG_\nu,\tcG=\sqcup_{\nu\in\bbN\bfI}\tcG_\nu$.

\subsection{Existence of Hall polynomials}

Let $\cG''$ be the set of all pairs $(\bfc,\underline{\lz})$ where $\bfc\in\cG'$ and $\underline{\lz}=(\lz_1,\dots,\lz_t)$ is a sequence of partitions. Here we allow $\lz_i=0$ for some $i$.

Let $\underline{d}=(d_1,\dots,d_t)$ be a sequence of integers and $\underline{z}=(\lz_1,\dots,\lz_t)$ be a pairwise-different sequence in $\cZ$. We say that $\underline{z}$ is of type $\underline{d}$ if $\deg z_i=d_i$ for all $i$. Let 
$$M_k(\bfc,\underline{\lz},\underline{z})=M_k(\bfc)\oplus M_k(\underline{\lz},\underline{z}).$$

\begin{lemma}[\cite{Deng_Han}]\label{lem: exist of Hall poly}
Fix $\underline{d}$. Given $\bm{\az},\bm{\bz},\bm{\gz}\in\cG''$, there exists a polynomial $\varphi^{\bm{\az}}_{\bm{\bz},\bm{\gz}}\in\bbQ[T]$ such that for each finite field $k=\bbF_q$,
$$\varphi^{\bm{\az}}_{\bm{\bz},\bm{\gz}}(q)=g^{M_k(\bm{\az},\underline{z})}_{M_k(\bm{\bz},\underline{z}),M_k(\bm{\gz},\underline{z})}$$
holds for all $\underline{z}$ of type $\underline{d}$.
\end{lemma}

In particular, we see that $g^{M_k(\bfc_1)}_{M_k(\bfc_2),M_k(\bfc_3)}$ is a polynomial in $|k|$ with rational coefficients for $\bfc_1,\bfc_2,\bfc_3\in\cG'$.

\subsection{The order $\preceq$}

First, we define the lexicographic orderings on $\cG_-$ and $\cG_+$ as follow.
Given $\bfc_-,\bfd_-\in\cG_-$,
we say $\bfc_->_L\bfd_-$ if there exists $t\leq 0$ such that $\bfc_-(t)>\bfd_-(t)$ and $\bfc_-(l)=\bfd_-(l)$ for all $t< l\leq0$.
Given $\bfc_+,\bfd_+\in\cG_+$,
we say $\bfc_+>_L\bfd_+$ if there exists $t>0$ such that $\bfc_+(t)>\bfd_+(t)$ and $\bfc_-(l)=\bfd_-(l)$ for all $0<l<t$.

Second, there is a partial order $\leq_G$ on $\Pi_r$ for each $r\geq 2$. For $\pi,\pi'\in\Pi_r$, $\pi\leq_G\pi'$ if and only if $\udim M_k(\pi)=\udim M_k(\pi')$ and 
$$\dim_k\Hom_{\cK_r(k)}(M,M_k(\pi))\geq\dim_k\Hom_{\cK_r(k)}(M,M_k(\pi'))$$
for all $M\in\cK_r(k)$. There are several equivalent definitions for $\leq_G$ and one can refer to \cite{Bongartz_On_degenerations_and_extensions_of_finite-dimensional_modules} for details. 

The order $\leq_G$ is independent of the field $k$, so we can define a partial order, still denoted by $\leq_G$, on $\cG_0$ such that $\bfc_0=(\pi_1,\pi_2,\dots,\pi_s)<_G\bfc_0'=(\pi_1',\pi_2',\dots,\pi_s')$ if and only if $\pi_t\leq_G\pi_t'$ for all $1\leq t\leq s$ but not all equalities hold.

\begin{definition}\label{order}
Define a partial order $\preceq$ on $\tilde{\cG}_{\nu}$ by letting $(\bfc',m')\prec(\bfc,m)$ if and only if one of the followings is satisfied:
\begin{enumerate}
\item[(a)] $\bfc'_\pm>_L \bfc_\pm$, that is, by definition, $\bfc'_-\geq_L \bfc_-$ and $\bfc'_+\geq_L \bfc_+$ but not all equalities hold;

\item[(b)] $\bfc'_-= \bfc_-$, $\bfc'_+=\bfc_+$, $m'<m$;

\item[(c)] $\bfc'_-= \bfc_-$, $\bfc'_+=\bfc_+$, $m'=m$, and $\bfc'_0<_G\bfc_0$. 
\end{enumerate}
Define a partial order of $\cG_{\nu}$, also denoted by $\preceq$, by letting $(\bfc',t_{\lz'})\prec(\bfc,t_\lz)$ if and only if $(\bfc',|\lz'|)\prec(\bfc,|\lz|)$ in $\tilde{\cG}_{\nu}$, or $(\bfc',|\lz'|)=(\bfc,|\lz|)$ and $\lz'>\lz$ where the order of partitions is the lexicographic order (which is also the order of Specht modules, see \cite{Macdonald_Symmetric_functions_and_Hall_polynomials}).

\end{definition}

\section{The algebra $\cH^0(\ggz)$}

\subsection{Definition of $\cH^0(\ggz)$ and PBW basis}\label{sec def of h0}

Let $\cA'=\bbQ[v,v^{-1}]$. Following \cite{XXZ_Tame_quivers_and_affine_bases_I}, we define an $\cA'$-subalgebra $\cH^0(\ggz)$ of $\cH^*(\ggz)$ (see Section \ref{sec: RH alg}) as follow:
Consider the element $[M(\bfc_0)]$ of $\cH^*(\ggz)$ whose $k$-component is $[M_k(\bfc_0)]$ for each $\bfc_0\in\cG_0$, then $\cH^0(\ggz)$ is defined as subalgebra of $\cH^*(\ggz)$ generated by $\bbQ, v, v^{-1}, u_\bfi,\bfi\in\bfI$ and $[M(\bfc_0)],\bfc_0\in\cG_0$.

For $m\in\bbN$, consider the element $H_m$ in $\cH^*(\ggz)$ whose $k$-component is 
$$(H_m)_k=\sum_{[M]: M\in\cR^\ggz_{\rm h},\udim M=m\dz}v_k^{-\dim_k M}[M].$$
For a partition $\lz=(\lz_1,\lz_2,\dots,\lz_s)$, define 
$$H_\lz=\prod_{1\leq t\leq s}H_{\lz_t}$$
and 
$$S_\lz=\det(H_{\lz_t-t+t'})_{1\leq t,t'\leq s}.$$

For $\bfc_-\in\cG_-,\bfc_0\in\cG_0,\bfc_+\in\cG_+$, consider the elements $\lan M(\bfc_-)\ran,\lan M(\bfc_0)\ran,\lan M(\bfc_+)\ran$ in $\cH^*(\ggz)$ whose $k$-components are $\lan M_k(\bfc_-)\ran,\lan M_k(\bfc_0)\ran,\lan M_k(\bfc_+)\ran$, respectively. 

For $\bfc=(\bfc_-,\bfc_0,\bfc_+)\in\cG', t_\lz\in\bbP$, define 
$$N(\bfc,t_\lz)=\lan M(\bfc_-)\ran*\lan M(\bfc_0)\ran*S_\lz*\lan M(\bfc_+)\ran\in\cH^*(\ggz).$$

\begin{proposition}\label{multiplication of N}
Let $(\bfc^1,t_{\lz^1}),(\bfc^2,t_{\lz^2})\in\cG$. Then there exists a polynomial $\psi^{(\bfc,t_{\lz})}_{(\bfc^1,t_{\lz^1}),(\bfc^2,t_{\lz^2})}\in\cA'$ for each $(\bfc,t_\lz)\in\cG$ with $\bfc_-\geq_L \bfc^1_-,\bfc_+\geq_L \bfc^2_+$, such that the following equalities hold in $\cH^*(\ggz)$:
$$N(\bfc^1,t_{\lz^1})*N(\bfc^2,t_{\lz^2})=\sum_{(\bfc,t_\lz)\in\cG,\bfc_-\geq_L \bfc^1_-,\bfc_+\geq_L \bfc^2_+}\psi^{(\bfc,t_\lz)}_{(\bfc^1,t_{\lz^1}),(\bfc^2,t_{\lz^2})}N(\bfc,t_\lz).$$
\end{proposition}
\begin{proof}
See [\cite{XXZ_Tame_quivers_and_affine_bases_I}, proposition 8.3]. The proof is similar, except that here we use Lemma \ref{lem: exist of Hall poly} to show the coefficients are polynomials in $v$, therefore $\psi^{(\bfc,t_{\lz})}_{(\bfc^1,t_{\lz^1}),(\bfc^2,t_{\lz^2})}\in\cA'$.
\end{proof}

\begin{corollary}
The set $\{N(\bfc,t_\lz)|(\bfc,t_\lz)\in\cG\}$ is an $\cA'$-basis of $\cH^0(\ggz)$.
\end{corollary}

We call $\{N(\bfc,t_\lz)|(\bfc,t_\lz)\in\cG\}$ the PBW basis of $\cH^0(\ggz)$.

Moreover, we have the tensor-product decomposition
$$\cH^0(\ggz)=\cH^0(\cP)\otimes_{\cA'}\cH^0(\cT_1)\otimes_{\cA'}\cH^0(\cT_2)\otimes_{\cA'}\dots\otimes_{\cA'}\cH^0(\cT_s)\otimes_{\cA'}\cA'[H_1,H_2,\dots]\otimes_{\cA'}\cH^0(\cI)$$
as $\cA'$-spaces. Here $\cH^0(-)$ are defined as the $\cA'$-subalgebras for the corresponding components.

\subsection{Almost orthogonality}\label{sec: Almost orthogonality}

Fix a valued quiver $\ggz$ and its Cartan datum $(\bfI,(-,-))$, let $\bfU$ be the associated quantum group.
Define an algebra structure on $\bfU^+\otimes\bfU^+$ by 
$$(x_1\otimes x_2)(y_1\otimes y_2)=v^{(|x_2|,|y_1|)}x_1y_1\otimes x_2y_2,$$
where $|x|\in\bbN\bfI$ is the grading of homogeneous $x$ on $\bfU^+$ such that $|E_\bfi|=\bfi$.
Let $r:\bfU^+\ra\bfU^+\otimes\bfU^+$ be the the algebra homomorphism such that 
$r(E_\bfi)=E_\bfi\otimes 1+1\otimes E_\bfi.$

The operators ${}_\bfi r, r_\bfi:\bfU^+\ra\bfU^+$ are defined such that 
$r(x)=u_\bfi\otimes {}_\bfi r(x)$
plus terms of other bi-homogeneities, and $r(x)=r_\bfi (x)\otimes u_\bfi$
plus terms of other bi-homogeneities. 
One can deduce that 
$${}_\bfi r(xy)={}_\bfi r(x)y+v^{(\bfi,|x|)} x{}_\bfi r(y),r_\bfi(xy)=v^{(\bfi,|y|)} r_\bfi(x)y+ x r_\bfi(y)$$ 
for homogeneous $x,y\in\bfU^+$.

By \cite{Lusztig_Introduction_to_quantum_groups}, there is a non-degenerate symmetric bilinear form $(-,-):\bfU^+\times\bfU^+\ra\bbQ(v)$ such that
$$(E_\bfi,E_\bfj)=\dz_{\bfi\bfj}(1-v_\bfi^{-2})^{-1},(x,y_1y_2)=(r(x),y_1\otimes y_2),(x_1x_2,y)=(x_1\otimes x_2,r(y))$$
where the form on $\bfU^+\otimes\bfU^+$ is defined by $(x\otimes x',y\otimes y')=(x,x')(y,y')$.

For a fixed finite field $k$, by  \cite{Green_Hall_algebras_hereditary_algebras_and_quantum_groups}, there is an inner product $(-,-):\cH^*_k(\ggz)\times\cH^*_k(\ggz)\ra\bbQ(v_k)$ defined by
$$(\lan M\ran,\lan N\ran)=\dz_{MN}\frac{v_k^{2\dim_k\End_\ggz M}}{a_M},$$
where $a_M=|\Aut_\ggz M|$.
This bilinear form is also well-defined on $\cC^*(\ggz)$ and coincides with the form $(-,-)$ on $\bfU^+$ via the isomorphism $\bfU^+\cong\cC^*(\ggz)$. 

By \cite{Ringel_Green's_Theorem_on_Hall_Algebras} and \cite{Green_Hall_algebras_hereditary_algebras_and_quantum_groups}, there is a linear operator
$r:\cH^*_k(\ggz)\ra\cH^*_k(\ggz)\times\cH^*_k(\ggz)$ defined by
$$r([L])=\sum_{[M],[N]}v_k^{\lan \udim M,\udim N\ran}g_{MN}^L\frac{a_Ma_N}{a_L}[M]\otimes[N],$$
which is also well-defined on $\cC^*(\ggz)$ and coincides with the operator $r$ on $\bfU^+$ via the isomorphism $\bfU^+\cong\cC^*(\ggz)$.

Note that for any $\bfc_0\in\cG_0$, let $L=M_k(\bfc_0)$, then $g^L_{MN}\neq 0$ implies that $M,N$ have no direct summand in $\cR^\ggz_{\rm h}$, therefore $r([M(\bfc_0)])$ is well-defined. Since $r$ is an algebra homomorphism, $r$ is well-defined on $\cH^0(\ggz)$.

The operators ${}_\bfi r, r_\bfi: \cH^0(\ggz)\ra\cH^0(\ggz)$ are similarly well-defined, and their restriction on $\cC^*(\ggz)$ coincide with ${}_\bfi r, r_\bfi:\bfU^+\ra\bfU^+$, respectively.

\begin{proposition}
For $(\bfc,t_\lz),(\bfc',t_{\lz'})\in\cG$, we have $$(N(\bfc,t_\lz),N(\bfc',t_{\lz'}))\in\dz_{\bfc\bfc'}\dz_{\lz\lz'}+v^{-1}\bbQ[[v^{-1}]]\cap\bbQ(v).$$
\end{proposition}
\begin{proof}
The proof is similar to Proposition 8.8 in \cite{XXZ_Tame_quivers_and_affine_bases_I}.
\end{proof}

Therefore we call the PBW basis $\{N(\bfc,t_\lz)|(\bfc,t_\lz)\in\cG\}$ almost orthogonal.

\section{Monomial basis, PBW basis and a bar-invariant basis}\label{sec affine bases}

The construction of the monomial basis, the PBW basis and a bar-invariant basis has basically no difference from the (non-valued) quiver cases. In this section, all proofs are similar to those in Section 9 in \cite{XXZ_Tame_quivers_and_affine_bases_I}.

\subsection{Monomials}\label{section: monomials}
Let $\cS(\ggz)$ be the set of all pairs $\oz=(\underline{\bfi},\underline{a})$, where $\underline{\bfi}=(\bfi_1,\bfi_2,\dots,\bfi_m)$ is a sequence in $\bfI$ and $\underline{a}=(a_1,a_2,\dots,a_m)$ is a sequence in $\bbZ^{>0}$ of the same length. 
For such a pair $\oz=(\underline{\bfi},\underline{a})$, we define the corresonding monomial
$$\fkm^\oz=u_{\bfi_1}^{(a_1)}*u_{\bfi_2}^{(a_2)}*\dots*u_{\bfi_m}^{(a_m)}$$
in $\cC^*(\ggz)$.

Consider the quiver $Q$ such that $\ggz=\ggz(Q,\sz)$. Let $\cS(Q)$ be the set of all pairs $\oz'=(\underline{i},\underline{a})$, where $\underline{i}=(i_1,i_2,\dots,i_m)$ is a sequence in $I$ and $\underline{a}=(a_1,a_2,\dots,a_m)$ is a sequence in $\bbZ^{>0}$ of the same length. Then there is a injection $\cS(\ggz)\ra\cS(Q),\oz\mapsto\hat{\oz}$, where $\hat\oz$ is obtained from $\oz=(\underline{\bfi},\underline{a})$ by replacing each $\bfi_t$ by a sequence of all vertices (under any fixed order) inside $\bfi_t$, and repeating each $a_t$ by $|\bfi_t|$ times. 
Then we have a monomial 
$$\fkm^\hoz=\prod_{i\in\bfi_1}u_i^{(a_1)}*\prod_{i\in\bfi_2}u_i^{(a_2)}*\dots*\prod_{i\in\bfi_t}u_i^{(a_t)}$$
in $\cC^*(Q)$. Here $u_i$ commutes with $u_j$ if $i,j$ are in the same $\sz$-orbit.

For a fixed valued quiver $\ggz$, since we only consider quivers without oriented cycles, let $\bfI=\{\mathbf{1},\mathbf{2},\dots,\mathbf{n}\}$ such that there are no arrows from $\bfi$ to $\bfj$ if $\bfi>\bfj$.

For a given $\nu=\sum_\bfi\nu_\bfi \bfi\in\bbN\bfI$, define a monomial in $\cC^*(\ggz)$:
$$\fkm^{\nu}=u_{\mathbf{1}}^{(\nu_\mathbf{1})}*u_\mathbf{2}^{(\nu_\mathbf{2})}*\dots*u_\mathbf{n}^{(\nu_\mathbf{n})}.$$

\subsection{Monomials for $\cP^\ggz$ and $\cI^\ggz$}

For $j\in\bbZ$, let $\bfe_j$ be the function $\bbZ^{\leq 0}\ra\bbN$ for $j\leq 0$ or respectively $\bbZ^{>0}\ra\bbN$ for $j>0$, which takes value 1 on $j$ and 0 on others. Then for $m\in\bbN$, the entry of 
$N((m\bfe_j,0,0),0)\in\cH^0(\ggz)$ on $k$ is $\lan M_k(m\bfe_j)\ran=\lan M_k(\bmbz_j)^{\oplus m}\ran=\lan M_k(\bmbz_j)\ran^{(m)}$. Without ambiguity, we simply denote $N((\bfd,0,0),0)=N(\bfd,0)$ for $\bfd\in\cG-$ and $N((0,0,\bfd'),0)=N(\bfd',0)$ for $\bfd'\in\cG+$.
\begin{lemma}\label{monomial of indecomp of P}
For $m\geq 0$, $j\leq 0$, we have
$$\fkm^{m\bmbz_j}=N(m\bfe_j,0)+\sum_{(\bfc,t_\lz)\in\cG\atop\bfc_->_L m\bfe_j}\phi_{m\bfe_j}^{(\bfc,t_\lz)}(v) N(\bfc,t_\lz)$$
in $\cH^0(\ggz)$ for some $\phi_{m\bfe_j}^{(\bfc,t_\lz)}\in \cA'$.
\end{lemma}
\begin{lemma}\label{monomial of indecomp of I}
For $m\geq 0$, $j>0$, we have
$$\fkm^{m\bmbz_j}=N(m\bfe_j,0)+\sum_{(\bfc,t_\lz)\in\cG\atop\bfc_+>_L m\bfe_j}\phi_{m\bfe_j}^{(\bfc,t_\lz)}(v) N(\bfc,t_\lz)$$
in $\cH^0(\ggz)$ for some $\phi_{m\bfe_j}^{(\bfc,t_\lz)}\in \cA'$.
\end{lemma}

Consider $\bfd:\bbZ^{\leq 0}\ra\bbN$ in $\cG_-$ and $\bfd':\bbZ^{>0}\ra\bbN$ in $\cG_+$.
Write $\bfd=\sum_{j\leq 0}\bfd(j)\bfe_j$ and $\bfd'=\sum_{j>0}\bfd'(j)\bfe_j$.
Define the following monomials in $\cC^*(\ggz)$: 
$$\fkm^{\oz(\bfd)}=\fkm^{\bfd(0)\bmbz_0}*\fkm^{\bfd(-1)\bmbz_{-1}}*\fkm^{\bfd(-2)\bmbz_{-2}}*\cdots,$$
$$\fkm^{\oz(\bfd')}=\cdots*\fkm^{\bfd'(3)\bmbz_3}*\fkm^{\bfd'(2)\bmbz_{2}}*\fkm^{\bfd'(1)\bmbz_{1}}.$$
Note that these are finite products since $\bfd,\bfd'$ are finitely supported.

\begin{lemma}\label{monomial for proj}
For $\bfd\in\cG_-$, we have
$$\fkm^{\oz(\bfd)}=N(\bfd,0)+\sum_{(\bfc,t_\lz)\in\cG\atop\bfc_->_L\bfd}\phi_\bfd^{(\bfc,t_\lz)}(v) N(\bfc,t_\lz)$$
in $\cH^0(\ggz)$ for some $\phi_\bfd^{(\bfc,t_\lz)}\in\cA'$.
\end{lemma}

\begin{lemma}\label{monomial for inj}
For $\bfd'\in\cG_+$, we have
$$\fkm^{\oz(\bfd')}=N(\bfd',0)+\sum_{(\bfc,t_\lz)\in\cG\atop\bfc_+>_L\bfd'}\phi_{\bfd'}^{(\bfc,t_\lz)}(v) N(\bfc,t_\lz)$$
in $\cH^0(\ggz)$ for some $\phi_{\bfd'}^{(\bfc,t_\lz)}\in\cA'$.
\end{lemma}

\subsection{Monomials for $\cR_{\rm nh}^\ggz$}

Consider a fixed non-homogeneous tube $\cT$ first.
 We simply denote $r=\rk \cT$ and $\Pi=\Pi_r$.
Let $\vez:\cK_{r}\cong\cT$ be the equivalence where $\cK_r=\cK_r(k')$ for some extension field $k'$ of $k$.
Let $L_j=\vez(S_j)$ for $1\leq j\leq r$ where $S_j,1\leq j\leq r$ are the simple objects in $\cK_r$.

\begin{lemma}\label{monomial in tube of indecomp}
For $1\leq j\leq r$ and $a>0$, we have
$$\fkm^{a\udim L_j}=\lan L_j\ran^{(a)}+\sum_{(\bfc,t_\lz)\in\cG\atop\bfc_->0,\bfc_+>0}\phi^{(\bfc,t_\lz)}_{L_j,a}(v)N(\bfc,t_\lz)$$
in $\cH^0(\ggz)$ for some $\phi^{(\bfc,t_\lz)}_{L_j,a}\in\cA'$.
\end{lemma}

By \cite{Deng_Du_Xiao_Generic_extensions_and_canonical_bases_for_cyclic_quivers}, there is a multiplication $\diamond$ on $\Pi^a$ such that $M_{\cK_r}(\pi\diamond\pi')$ is the generic extension of $M_{\cK_r}(\pi)$ by $M_{\cK_r}(\pi')$. Moreover, for any $\pi\in\Pi^a$, there exists
$1\leq j_1,j_2,\dots,j_s\leq r$ and $a_1,a_2,\dots,a_s\geq 1$ such that
$$\pi=(a_1[i_1;1))\diamond (a_2[i_2;1))\diamond \dots(a_s[i_s;1)),$$
then
$$\lan S_{j_1}\ran^{(a_1)}*\lan S_{j_2}\ran^{(a_2)}*\dots*\lan S_{j_s}\ran^{(a_s)}=\lan M_{\cK_r}(\pi)\ran+\sum_{\pi'\in\Pi,\pi'<_G \pi}\xz^{\pi'}_{\oz_\pi}(v)\lan M_{\cK_r}(\pi')\ran$$
in $\cH^*(\cK_r)$.
Applying $\vez$ to the both sides, we get
$$\lan L_{j_1}\ran^{(a_1)}*\lan L_{j_2}\ran^{(a_2)}*\dots*\lan L_{j_s}\ran^{(a_s)}=\lan M_{\cT}(\pi)\ran+\sum_{\pi'\in\Pi,\pi'<_G \pi}\xz^{\pi'}_{\oz_\pi}(v)\lan M_{\cT}(\pi')\ran$$
in $\cH^*(\cT)$.
Define a monomial $\fkm^{\oz(\pi)}$ in $\cC^*(\ggz)$:
$$\fkm^{\oz(\pi)}=\fkm^{a_1\udim L_{i_1}}*\fkm^{a_2\udim L_{i_2}}*\dots*\fkm^{a_s\udim L_{i_s}}.$$

\begin{lemma}\label{monomial in tube}
For $\pi\in\Pi^a$, we have
$$\fkm^{\oz(\pi)}=\lan M_{\cT}(\pi)\ran+\sum_{\pi'\in\Pi,\pi'<_G \pi}\xz^{\pi'}_{\pi}(v)\lan M_{\cT}(\pi')\ran+\sum_{(\bfc,t_\lz)\in\cG\atop\bfc_->0,\bfc_+>0}\phi^{(\bfc,t_\lz)}_{\pi}(v)N(\bfc,t_\lz)$$
in $\cH^0(\ggz)$ for some $\xz^{\pi'}_{\pi},\phi^{(\bfc,t_\lz)}_{\pi}\in\cA'$.
\end{lemma}

For $\bfc_0=(\pi_1,\pi_2,\cdots,\pi_s)\in\Pi^a(1)\times\Pi^a(2)\times\cdots\times\Pi^a(s)$,
define
$$\fkm^{\oz(\bfc_0)}=\fkm^{\oz({\pi_1})}*\fkm^{\oz(\pi_2)}*\cdots*\fkm^{\oz(\pi_s)}.$$

\begin{lemma}\label{monomial for nh}
For $\bfc_0\in\cG^a_0$, we have
$$\fkm^{\oz(\bfc_0)}=\lan M(\bfc_0)\ran+\sum_{\bfc'_0\in\cG_0\atop\bfc'_0<_G\bfc_0}\phi^{\bfc'_0}_{\bfc_0}(v)\lan M(\bfc'_0)\ran+\sum_{(\bfc,t_\lz)\in\cG\atop\bfc''_->0,\bfc''_+>0}\phi^{(\bfc'',t_\lz)}_{\bfc_0}(v)N(\bfc'',t_\lz),$$
in $\cH^0(\ggz)$ for some $\phi^{\bfc'_0}_{\bfc_0},\phi^{(\bfc'',t_\lz)}_{\bfc_0}\in\cA'$.
\end{lemma}

\subsection{Monomials for homogeneous regular part}

\begin{lemma}\label{monomial for P}
For $m\geq 0$, we have
$$\fkm^{m\dz}=H_m+\sum_{(\bfc,t_\lz)\in\cG\atop\bfc_\pm=0,\bfc_0\neq 0}\phi^{(\bfc,t_\lz)}_{m\dz}(v)N(\bfc,t_\lz)+\sum_{(\bfc,t_\lz)\in\cG\atop\bfc_->0,\bfc_+>0}\phi^{(\bfc,t_\lz)}_{m\dz}(v)N(\bfc,t_\lz),$$
in $\cH^0(\ggz)$ for some $\phi^{(\bfc,t_\lz)}_{m\dz}\in\cZ$.
\end{lemma}

For a partition $\lz=(\lz_1,\lz_2,\cdots,\lz_t)$, define
$$\fkm^{\oz(t_\lz)}=\fkm^{\lz_1\dz}*\fkm^{\lz_2\dz}*\cdots*\fkm^{\lz_t\dz}.$$
When $\lz=0=t_\lz$, set $\fkm^{\oz(0)}=1$.
\begin{lemma}\label{monomial for h}
For a nonzero partition $\lz=(\lz_1,\lz_2,\cdots,\lz_t)$, we have
$$\fkm^{\oz(t_\lz)}=H_\lz+\sum_{(\bfc,t_\mu)\in\cG\atop\bfc_\pm=0,\bfc_0\neq 0}\phi^{(\bfc,t_\mu)}_{t_\lz}(v)N(\bfc,t_\mu)+\sum_{(\bfc,t_\mu)\in\cG\atop\bfc_->0,\bfc_+>0}\phi^{(\bfc,t_\mu)}_{t_\lz}(v)N(\bfc,t_\mu),$$
in $\cH^0(\ggz)$ for some $\phi^{(\bfc,t_\mu)}_{t_\lz}\in\cA'$.
\end{lemma}

\subsection{The monomial basis}

For $(\bfc,t_\lz)\in\cG^a$, we define the monomial
$$\fkm^{\oz(\bfc,t_\lz)}=\fkm^{\oz(\bfc_-)}*\fkm^{\oz(\bfc_0)}*\fkm^{\oz(t_\lz)}*\fkm^{\oz(\bfc_+)}.$$

\begin{proposition}\label{monomial basis}
For $(\bfc,t_\lz)\in\cG^a$, we have
$$\fkm^{\oz(\bfc,t_\lz)}=N(\bfc,t_\lz)+\sum_{(\bfc',t_{\lz'})\in\cG\atop(\bfc',t_{\lz'})\prec(\bfc,t_\lz)}\phi_{(\bfc,t_\lz)}^{(\bfc',t_{\lz'})}(v)N(\bfc',t_{\lz'})$$
in $\cH^0(\ggz)$ for some $\phi_{(\bfc,t_\lz)}^{(\bfc',t_{\lz'})}\in\cA'$, such that $\phi_{(\bfc,t_\lz)}^{(\bfc,t_\mu)}=K_{\lz\mu},\mu\geq\lz$, where $K_{\lz\mu}$ is the Kostka number.
Moreover, $\{\fkm^{\oz(\bfc,t_\lz)}|(\bfc,t_\lz)\in\cG^a\}$ is an $\cA'$-basis of $\cC^*_{\cA'}(\ggz)$.
\end{proposition}

The basis $\{\fkm^{\oz(\bfc,t_\lz)}|(\bfc,t_\lz)\in\cG^a\}$ is called the monomial basis.

\subsection{The PBW basis}

Using the method in \cite{Deng_Du_Xiao_Generic_extensions_and_canonical_bases_for_cyclic_quivers}, a set $\{E(\bfc,t_\lz)|(\bfc,t_\lz)\in\cG^a\}$ can be constructed inductively to satisfy the following proposition.

\begin{proposition}\label{uniqueness of PBW}
For each $(\bfc,t_\lz)$ in $\cG^a$, there is a unique element $E(\bfc,t_\lz)\in\cC^*(\ggz)$ such that
$$E(\bfc,t_\lz)=N(\bfc,t_\lz)+\sum_{(\bfc',t_{\lz'})\in\cG\setminus\cG^a\atop (\bfc',t_{\lz'})\prec(\bfc,t_\lz)}b_{(\bfc,t_\lz)}^{(\bfc',t_{\lz'})}(v)N(\bfc',t_{\lz'})$$
holds in $\cH^0(\ggz)$ for some $b_{(\bfc,t_\lz)}^{(\bfc',t_{\lz'})}\in\cA'$.
Moreover, 
$$\fkm^{\oz(\bfc,t_\lz)}=E(\bfc,t_\lz)+\sum_{(\bfc',t_{\lz'})\in\cG^a\atop(\bfc',t_{\lz'})\prec(\bfc,t_\lz)}\phi_{(\bfc,t_\lz)}^{(\bfc',t_{\lz'})}(v)E(\bfc',t_{\lz'})$$
holds in $\cC^*(\ggz)$ for some $\phi_{(\bfc,t_\lz)}^{(\bfc',t_{\lz'})}\in\cA'$.
Moreover, $\{E(\bfc,t_\lz)|(\bfc,t_\lz)\in\cG^a\}$ is an $\cA'$-basis of $\cC^*_{\cA'}(\ggz)$.
\end{proposition}

The basis $\{E(\bfc,t_\lz)|(\bfc,t_\lz)\in\cG^a\}$ is called the PBW basis.

\subsection{A bar-invariant basis}\label{sec bar-inv basis}

Recall that $\bar{}:\bfU^+\ra\bfU^+$ is the bar-involution such that $\bar{E_\bfi}=E_\bfi$ and $\bar{v}=v^{-1}$.

\begin{proposition}\label{relation bt C and E}
For each $(\bfc,t_\lz)$ in $\cG^a$, there is a unique element $C(\bfc,t_\lz)\in\cC_{\cA'}^*(\ggz)$ such that 
$\overline{C{(\bfc,t_\lambda)}}=C{(\bfc,t_\lambda)}$ and
$$C(\bfc,t_\lz)=E(\bfc,t_\lz)+\sum_{(\bfc',t_{\lz'})\in\cG^a,(\bfc',t_{\lz'})\prec(\bfc,t_\lz)}g^{(\bfc',t_{\lz'})}_{(\bfc,t_\lz)}(v)E(\bfc',t_{\lz'}).$$
holds in $\cC^*(\ggz)$ for some $g^{(\bfc',t_{\lz'})}_{(\bfc,t_\lz)}\in v^{-1}\bbQ[v^{-1}]$.
Moreover, $\{C(\bfc,t_\lz)|(\bfc,t_\lz)\in\cG^a\}$ is an $\cA'$-basis of $\cC^*_{\cA'}(\ggz)$.
\end{proposition}

\begin{corollary}\label{relation bt C and m}
    For $(\bfc,t_\lz)\in\cG^a$, we have 
    $$\fkm^{\oz(\bfc,t_\lz)}=C(\bfc,t_\lz)+\sum_{(\bfc',t_{\lz'})\in\cG^a,(\bfc',t_{\lz'})\prec(\bfc,t_\lz)}h^{(\bfc',t_{\lz'})}_{(\bfc,t_\lz)}(v) C(\bfc',t_{\lz'})$$
    holds in $\cC^*(\ggz)$ for some $h^{(\bfc',t_{\lz'})}_{(\bfc,t_\lz)}\in\cA'$.
\end{corollary}

\subsection{A property of the PBW basis}

We prove here a property of the PBW basis $\{E(\bfc,t_\lz)|(\bfc,t_\lz)\in\cG^a\}$.

We identify $\bfc_-,\bfc_0,\bfc_+$ with $(\bfc_-,0,0),(0,\bfc_0,0),(0,0,\bfc_+)$. Also, we simply denote $E(\bfc_-)=\lan M(\bfc_-)\ran,E(\bfc_+)=\lan M(\bfc_+)\ran$ and $E(\bfc_0,t_\lz)=E((0,\bfc_0,0),t_\lz)$ for $(\bfc,t_\lz)\in\cG^a$.

\begin{lemma}
For $(\bfc_0,t_\lz)\in\cG_0^a\times\bbP$, we have 
$$E(\bfc_0,t_\lz)=N(\bfc_0,t_\lz)+\sum_{(\bfc'_0,t_{\lz'})\in(\cG_0\setminus\cG_0^a)\times\bbP\atop(\bfc'_0,t_{\lz'})\prec(\bfc_0,t_{\lz})}b_{(\bfc_0,t_\lz)}^{(\bfc'_0,t_{\lz'})}(v)N(\bfc'_0,t_{\lz'})$$
for some $b_{(\bfc_0,t_\lz)}^{(\bfc'_0,t_{\lz'})}\in\cA'$. In other words, $E(\bfc_0,t_\lz)\in\cC^*\cap\cH^*(\cR^\ggz)$.
\end{lemma}

\begin{proof}
Write $\bfc_0=(\pi_1,\pi_2,\dots,\pi_s)\in\cG^a_0$. Consider
 $$\tilde{E}(\bfc_0)=\vez_1(E_{\pi_1})*\vez_2(E_{\pi_2})*\dots*\vez_s(E_{\pi_s}),$$
where $E_{\pi_t}$ is the PBW basis element of $\cC^*(\cK_{r_t}(k_t))$ defined in \cite{Deng_Du_Xiao_Generic_extensions_and_canonical_bases_for_cyclic_quivers}, satisfying
$$E_\pi-\lan M(\pi)\ran\in\sum_{\pi'\in\Pi_{ lr_t}\setminus\Pi_{r_t}^a\atop\pi'<_G\pi}\cA\lan M(\pi')\ran,\forall \pi\in\Pi_{r_t}^a.$$
Since $\vez_t$ sends simple modules to exceptional modules, $\tilde{E}(\bfc_0)$ is in the composition algebra $\cC^*(\ggz)$.
Also, by definition of the orderings, we have
$$\tilde{E}(\bfc_0)-\lan M(\bfc_0)\ran\in\sum_{\bfd_0\in\cG_0\setminus\cG_0^a\atop\bfd_0\prec\bfc_0}\cA\lan M(\bfd_0)\ran.$$

Consider the element $\tilde{H}_m$ in $\cH^*(\ggz)$ whose $k$-component is
$$(\tilde{H}_m)_k=\sum_{[M]: M\in\cR^\ggz,\udim M=m\dz}v_k^{-\dim_k M}[M].$$

In fact, we can prove that $\tilde{H}_m\in\cC^*(\ggz)$ by induction on $m$ and that
$$\fkm^{m\dz}=\Tilde{H}_m+\sum_{0\leq j<m}A_j\Tilde{H}_jB_j$$
for some $A_j\in\cC^*(\cP^\ggz)\subset\cC^*$ and $B_j\in\cC^*(\cI^\ggz)\subset\cC^*$. Here $\cC^*(\cP^\ggz),\cC^*(\cI^\ggz)$ are subalgebras of $\cH^*(\ggz)$ spanned by $\lan M(\bfc_-)\ran,\bfc_-\in \cG^-$ and $\lan M(\bfc_+)\ran,\bfc_+\in \cG^+$, respectively.

Let $$\tilde{S}_\lz=\det(\tilde{H}_{\lz_t-t+t'})_{1\leq t,t'\leq s}$$ for a partition $\lz$.
By definition, we have $$\Tilde{H}_m-H_m\in\sum_{(\bfc_0,t_\mu),|\mu|<m}\cA N(\bfc_0,t_\mu).$$
Then
$$\Tilde{S}_\lz-S_\lz\in\sum_{(\bfc_0,t_\mu),|\mu|<m}\cA N(\bfc_0,t_\mu).$$

Now for $(\bfc_0,t_\lz)\in\cG^a_0\times\bbP$, consider $$\tilde{E}(\bfc_0,t_\lz)=\tilde{E}(\bfc_0)*\tilde{S_\lz}\in\cC^*\cap\cH^*(\cR^\ggz),$$
then $$\tilde{E}(\bfc_0,t_\lz)-N(\bfc_0,t_\lz)\in\sum_{(\bfd_0,t_\mu)\in\cG_0\times\bbP\atop(\bfd_0,t_\mu)\prec(\bfc_0,t_\lz)}\cA' N(\bfd_0,t_\mu).$$

Assume that this lemma holds for all
$(\bfc'_0,t_{\lz'})\in\cG_0^a\times\bbP$ such that 
$(\bfc'_0,t_{\lz'})\prec(\bfc_0,t_\lz)$.
Write
$$\tilde{E}(\bfc_0,t_\lz)=N(\bfc_0,t_\lz)+\sum_{(\bfd_0,t_\mu)\in\cG_0\times\bbP\atop(\bfd_0,t_\mu)\prec(\bfc_0,t_\lz)}\gz_{(\bfc_0,t_\lz)} ^{(\bfd_0,t_\mu)}N(\bfd_0,t_\mu),\gz^-_-\in\cA'$$
then by assumption 
$$\tilde{E}(\bfc_0,t_\lz)-\sum_{(\bfc'_0,t_{\lz'})\in\cG_0^a\times\bbP\atop(\bfc'_0,t_{\lz'})\prec(\bfc_0,t_\lz)}\gz_{(\bfc_0,t_\lz)} ^{(\bfc'_0,t_{\lz'})}E(\bfc'_0,t_{\lz'})-N(\bfc_0,t_\lz)\in\sum_{(\bfd_0,t_\mu)\in(\cG_0\setminus\cG_0^a)\times\bbP\atop(\bfd_0,t_\mu)\prec(\bfc_0,t_\lz)}\cA'N(\bfd_0,t_\mu).$$
Therefore by Proposition \ref{uniqueness of PBW}, we have
$$\tilde{E}(\bfc_0,t_\lz)-\sum_{(\bfc'_0,t_{\lz'})\in\cG_0^a\times\bbP\atop(\bfc'_0,t_{\lz'})\prec(\bfc_0,t_\lz)}\gz_{(\bfc_0,t_\lz)} ^{(\bfc'_0,t_{\lz'})}E(\bfc'_0,t_{\lz'})=E(\bfc_0,t_\lz).$$
The left side is clearly in $\cC^*\cap\cH^*(\cR^\ggz)$, so is the right side.

\end{proof}

The following proposition gives a more explicit interpretation for the PBW basis.

\begin{proposition}\label{PBW basis is multiplicative}
For $(\bfc,t_\lz)\in\cG^a$, we have
$$E(\bfc,t_\lz)=E(\bfc_-)*E(\bfc_0,t_\lz)*E(\bfc_+).$$
\end{proposition}

\begin{proof}
    The previous lemma implies that 
    $$E(\bfc_-)*E(\bfc_0,t_\lz)*E(\bfc_+)-N(\bfc,t_\lz)\in\sum_{(\bfc',t_{\lz'})\in\cG\setminus\cG^a\atop(\bfc',t_{\lz'})\prec(\bfc,t_{\lz})}\cA'N(\bfc',t_{\lz'}),$$
    then the proof is completed by Proposition \ref{uniqueness of PBW}.
\end{proof}

\section{The geometric realization and the canonical basis}

In this section, we recall the classical geometric realization introduced by Lusztig in \cite{Lusztig_Introduction_to_quantum_groups}.

\subsection{Perverse sheaves for quivers}
Let $Q=(I,H,s,t)$ be any quiver. Fix a finite field $k$ and a prime $l$  which is invertible in $k$. All algebraic varieties will be over the algebraic closure $\bar{k}$ of $k$. 

Let $V$ be an $I$-graded vector space and $\bbE_V=\oplus_{h\in H}\Hom_{\bar{k}}(V_{s(h)},V_{t(h)})$ be the representation variety and $\GL_V=\oplus_{i\in I}\GL_{\bar{k}}(V_i)$.
There is a $\GL_V$-action on $\bbE_V$ given by $(g.x)_h=g_{t(h)}x_hg^{-1}_{s(h)}$, such that each orbit contains isomorphic representations.

Given any $\nu\in\bbN I$, let $\cS_{\nu}(Q)$ be the subset of $\cS(Q)$ consisting of all
$$\oz'=(\underline{i},\underline{a})=((i_1,\dots,i_t),(a_1,\dots,a_t))$$
such that $\sum_ma_mi_m=\nu$.
Let $|V|=\sum_{i\in I}(\dime V_i)i$, for any $\oz\in\cS_{|V|}(Q)$, we say that a flag of $I$-graded subspaces
$$V^\bullet=(V=V^0\supseteq V^1\supseteq\cdots\supseteq V^t=0)$$
is of type $\oz'$ if $|V^{m-1}/V^m|=a_mi_m$ for all $m=1,2,\cdots,t$.

Let $\cF_{\oz'}$ be the flag variety consisting of all flags of type $\oz'$. Then $\GL_V$ acts on $\cF_{\oz'}$ naturally.

Given $x\in\bbE_V, V^\bullet\in\cF_{\oz'}$, we say that $V^\bullet$ is $x$-stable if $x_h(V^m_{s(h)})\subseteq V^m_{t(h)}$ for any $h\in H$ and $m=1,2,\cdots,t$.
Define $\tcF_{\oz'}$ to be the variety consisting of all pairs $(x,V^\bullet)\in\bbE_V\times\cF_{\oz'}$ such that $V^\bullet$ is $x$-stable. Then $\GL_V$ acts on $\tcF_{\oz'}$ naturally.

Let $\mathbf{1}_{\tcF_{\oz'}}$ be the constant sheaf on $\tcF_{\oz'}$.
Let $\pi_{\oz'}:\tcF_{\oz'}\ra\bbE_V$ be the first projection $(x,V^\bullet)\mapsto x$.
By the decomposition theorem of Beilinson, Bernstein and Deligne (\cite{BBD_Faisceaux_pervers}), the Lusztig sheaf (omitting the Tate twist)
$$L_{\oz'}:=(\pi_{\oz'})_!\mathbf{1}_{\tcF_{\oz'}}[\dime \tcF_{\oz'}]$$
is a $\GL_V$-equivariant semisimple complex on $\bbE_V$.

Let $\cD(\bbE_V)$ be the derived category of complexes of $l$-adic sheaves on $\bbE_V$.  Let $\cM_{\GL_V}(\bbE_V)$ be the full subcategory of $\cD(\bbE_V)$ consisting of $\GL_V$-equivariant perverse sheaves over $\bbE_V$. Let $\cP_V$ be the full subcategory of $\cM_{\GL_V}(\bbE_V)$ consisting of perverse sheaves which are direct sums of simple perverse sheaves $L$ that have the following property: $L[d]$ appears as a direct summand of $L_{\oz'}$ for some ${\oz'}\in\cS_{\nu}(Q)$. Let $\cQ_V$ be the full category of $\cD(\bbE_V)$ whose objects are the complexes that are isomorphic to finite direct sums of complexes of the form $L[d']$ for various simple perverse sheaves $L\in\cP_V$ and $d'\in\bbZ$.

By the construction of Lusztig (\cite{Lusztig_Introduction_to_quantum_groups}), the Grothendieck group $K(\cQ_V)$, which depends only on $|V|$, is a free $\cA$-module such that $v$ acts by shift. Write $K(\cQ_V)=K(\cQ)_{|V|}$. Let $K(\cQ)=\oplus_{\nu\in\bbN I}K(\cQ)_\nu$. Then $K(\cQ)$ becomes an $\cA$-algebra where the multiplication is defined by 
$$[L]\circ[L']=[\Ind_{T,W}^V(L\boxtimes L')].$$
Here $W\subset V, T=V/W, L\in\cQ_T, L'\in\cQ_W$ and $\Ind_{T,W}^V$ is the induction functor (see \cite{Lusztig_Introduction_to_quantum_groups}, Chapter 9).

\begin{lemma}\label{KQ isom to U+ for quiver}
Let $\bfU$ be the quantum group associated to the Cartan datum of $Q$. Let 
$\Xi_2:\bfU^+_\cA\ra K(\cQ)$ be an $\cA$-linear map such that 
$$\Xi_2(E_{i_1}^{(a_1)}*E_{i_2}^{(a_2)}*\dots*E_{i_t}^{(a_t)})=[L_{\oz'}]$$
for all ${\oz'}=((i_1,\dots,i_t),(a_1,\dots,a_t))\in\cS_\nu(Q)$ and $\nu\in\bbN I$. Then $\Xi_2$ is an isomorphism of $\cA$-algebras.
\end{lemma}

\subsection{Perverse sheaves for valued quivers}\label{section: Perverse sheaves for valued quivers}

Fix $\ggz=\ggz(Q,\sz)$, $\ggz=(\bfI,\bfH,s,t;d,m)$ and $Q=(I,H,s,t)$. Since $\bfI$ is the set of $\sz$-orbits in $I$, there is a bijection $\eta':\bbZ\bfI\ra(\bbZ I)^\sz$ such that $\eta'(\bfi)=\sum_{i\in\bfi}i$.

For an I-graded vector space $V$ such that $|V|\in(\bbZ I)^\sz$, by abuse of notation, choose an automorphism $\sz:V\ra V$ such that for all $i\in I$, $\sz(V_i)=V_{\sz(i)}$ and $\sz^j|_{V_i}=1_{V_i}$ if $\sz^j(i)=i$. 
Then it naturally induces isomorphisms $\sz:\GL_V\ra \GL_V$ and $\sz:\bbE_V\ra \bbE_V$ such that $\sz(g.z)=\sz(g).\sz(z)$ for all $g\in\GL_V,z\in V$ and $\sz\circ x_h=\sz(x)_{\sz(h)}\circ \sz$ for all $x\in\bbE_V,h\in H$. 
Taking the inverse image under $\sz:\bbE_V\ra\bbE_V$ gives us a functor $\sz^*:\cD(\bbE_V)\ra \cD(\bbE_V)$. Let $T>1$ be minimal such that $\sz^T=1_V$, then $\sz^{*T}=1_{\cD(\bbE_V)}$.

Define a category $\tcQ_V$ consisting of objects of the form $(L,\phi)$, where $L\in\cQ_V$ and $\phi:\sz^*L\ra L$ is an isomorphism in $\cD(\bbE_V)$, such that 
$$\phi\circ\sz^*\phi\circ\dots\circ\sz^{*(T-1)}\phi=1_L.$$
The morphisms in $\tcQ_V$ are those in $\cQ_V$ commuting with $\phi$'s.
We say an object in $\tcQ_V$ is traceless if it is isomorphic to $(L\oplus \sz^*L\oplus\dots\oplus \sz^{*(t-1)}L,\phi)$ for some $L\in\cQ_V$ and $t>1$ dividing $T$ such that $\sz^{*t}L\cong L$, where $\phi:\sz^*L\oplus \sz^{*2}L\oplus\dots\oplus \sz^{*t}L\ra L\oplus \sz^*L\oplus\dots\oplus \sz^{*(t-1)}L$ carrying the summand $\sz^{*j}L$ onto the summand $\sz^{*j}L$ for $1\leq j\leq t-1$ and the summand $\sz^{*t}L$ onto the summand $L$.

The Grothendieck group $K(\tcQ_V)$ is the $\cA$-module generated by symbols $[L,\phi]$, one for each isomorphism class of objects $(L,\phi)$ in $\tcQ_V$, subject to the following relations:
\begin{enumerate}
\item[(1)] $[L,\phi]+[L',\phi']=[L\oplus L',\phi\oplus\phi']$;
\item[(2)] If $(L,\phi)$ is traceless ,then $[L,\phi]=0$;
\item[(3)] $v^j[L,\phi]=[L[j],\phi[j]], j\in\bbZ$.
\end{enumerate}
Also, $K(\tcQ_V)$ depends only on $|V|$, so we write $K(\tcQ_V)=K(\tcQ)_{|V|}$ and define $K(\tcQ)=\oplus_{\nu\in\bbN\bfI}K(\tcQ)_{\eta'(\nu)}$.
The multiplication on $K(\tcQ)$ is given by 
$$[L',\phi']\circ[L'',\phi'']=[\Ind^V_{T,W}(L'\boxtimes L''),\phi]$$
for $W\subset V, T=V/W, (L',\phi')\in\cQ_T,(L'',\phi'')\in\cQ_W$, where $\phi$ is the composition of ${\rm Ind}^V_{T,W}(\phi'\boxtimes \phi'')$ and the natural isomorphism $\sz^*{\rm Ind}^V_{T,W}(L'\boxtimes L'')\cong{\rm Ind}^V_{T,W}\sz^*(L'\boxtimes L'')$.
This makes $K(\tcQ)$ an $\cA$-algebra.

Consider $\oz=((\bfi_1,\dots,\bfi_t),(a_1,\dots,a_t))\in\cS_{\nu}(\ggz),\nu\in\bbN \bfI$,  we say that a flag of $I$-graded subspaces
$$V^\bullet=(V=V^0\supseteq V^1\supseteq\cdots\supseteq V^t=0)$$
is of type $\oz$ if $|V^{m-1}/V^m|=a_m\eta'(\bfi_m)$ for all $m=1,2,\cdots,t$. 
Let $\cF_{\oz}$ be the flag variety consisting of all flags of type $\oz$. 

Note that for each $\bfi\in\bfI$, vertices in $\bfi$ are disjoint to each other, $\cF_{\oz}$ is actually isomorphic to $\cF_{\hat\oz}$, where $\hat\oz\in\cS_{\eta'(\nu)}(Q)$ is introduced in Section \ref{section: monomials} and $\cF_\hoz$ is the flag variety of type $\hoz$ for quiver $Q$.

Similarly, define $\tcF_\oz$ to be the variety consisting of all pairs $(x,V^\bullet)\in\bbE_V\times\cF_\oz$ such that $V^\bullet$ is $x$-stable. Then $\sz:V\ra V$ naturally induces isomorphisms $\sz:\cF_\oz\ra\cF_\oz$ and $\sz:\tcF_\oz\ra\tcF_\oz$. The obvious isomorphism $\sz^*\bfone_{\tcF_\oz}\cong\bfone_{\tcF_\oz}$ induces a canonical isomorphism 
$$\sz^*(\pi_\oz)_!\bfone_{\tcF_\oz}=(\pi_\oz)_!(\sz^*\bfone_{\tcF_\oz})\cong(\pi_\oz)_!\bfone_{\tcF_\oz}.$$
By applying to both sides a shift by $\dime \tcF_\oz$, we obtain an isomorphism $\sz^* L_\oz\ra L_\oz$ denoted by $\phi_0$, where $L_\oz=(\pi_\oz)_!\bfone_{\tcF_\oz}[\dime \tcF_\oz]$ is isomorphic to $L_{\hoz}$. 

\begin{lemma}\label{KQ isom to U+ for v quiver}
Let $\bfU$ be the quantum group associated to the Cartan datum of $\ggz$. Let
$\Xi_2:\bfU^+_\cA\ra K(\tcQ)$ be an $\cA$-linear map such that 
$$\Xi_2(E_{\bfi_1}^{(a_1)}*E_{\bfi_2}^{(a_2)}*\dots*E_{\bfi_t}^{(a_t)})=[L_{\oz},\phi_0]$$
for all $\oz=((\bfi_1,\dots,\bfi_t),(a_1,\dots,a_t))\in\cS_\nu$ and $\nu\in\bbN \bfI$. Then $\Xi_2$ is an isomorphism of $\cA$-algebras.
\end{lemma}

Easy to see that Lemma \ref{KQ isom to U+ for quiver} is a special case of  Lemma \ref{KQ isom to U+ for v quiver}.

\subsection{The canonical basis}\label{sec: canonical basis}

Let $\bfU$ be the quantum group of $\ggz$ of any type. The signed canonical basis $\cB$ is by definition 
the set of all $x\in\bfU^+_\cA$ such that $\bar{x}=x$ and $(x,x)\in 1+v^{-1}\bbZ[[v^{-1}]]$. By Theorem 14.2.3 in \cite{Lusztig_Introduction_to_quantum_groups}, for a subset $B\subset \cB$, $\cB=B\sqcup (-B)$ implies that $B$ is an $\cA$-basis of $\bfU^+_\cA$ and also a $\bbQ(v)$-basis of $\bfU^+$.

Let $D$ be the Vernier duality. The following lemma is from Proposition 12.5.2 in \cite{Lusztig_Introduction_to_quantum_groups}.
\begin{lemma}\label{lemma: CB is unique up to sign}
For each simple $L\in\cP_V$ such that $\sz^*L\cong L$, there exists an isomorphism $\phi:\sz^* L\ra L$ such that $(D(L),D(\phi)^{-1})\cong (L,\phi)$. Moreover, $\phi$ is unique up to a multiplication by $\pm 1$.
\end{lemma}

By Lusztig, under the isomorphism $\Xi_2: \bfU^+_\cA\cong K(\tcQ)$, the bar involution $x\mapsto \bar{x}$ coincides with the Verdier duality $[L,\phi]\mapsto[D(L),D(\phi)^{-1}]$, then 
the set $\cB''=\Xi(\cB)$ consists of all elements of the form $[L,\phi]$ where $(L,\phi)$ are as above lemma.
 
Given  $\bfi\in\bfI$ and $n\geq 0$, let $\cB_{\bfi;\geq n}=\cB\cap E_\bfi^n\bfU^+$ and $\cB_{\bfi;n}=\cB_{\bfi;\geq n}-\cB_{\bfi;\geq n+1}$. The canonical basis $\bfB\subset \cB$ is defined as 
$$\bfB=\{b\in\cB|\sgn(b)=1\},$$
where $\sgn:\cB\ra\{\pm 1\}$ is defined by properties below:
\begin{enumerate}
\item[(1)] $\sgn(1)=1$.
\item[(2)] For $b'\in\cB_{\bfi;0}$ and $n\geq 0$, there is a unique $b\in\cB_{\bfi;n}$ such that $E_\bfi^{(n)}b'=b$ plus an $\cA$-linear combination of elements in $\cB_{\bfi;\geq n+1}$. Then $\sgn(b')=\sgn(b)$.
\end{enumerate}
Also, the corresponding $b\mapsto b'$ is denoted by $\varphi_{\bfi;n}:\cB_{\bfi;0}\ra\cB_{\bfi;n}$, which is a bijection of sets.
We shall see that $\bfB$ is compatible with Kashiwara's operators in Section \ref{Sec: Kashi}.

\subsection{Simple perverse sheaves for tame quivers}\label{sec Simple perverse sheaves for tame quivers}

Now assume $\ggz=\ggz(Q,\sz)$ is tame.
We denote by $\cG(Q)$ (respectively, $\cG'(Q),\cG^a(Q),\tcG(Q)$) the index set $\cG$ (respectively, $\cG',\cG^a,\tcG$)  introduced in Section \ref{index set} for quiver $Q$ instead of the valued quiver $\ggz$.
Let $\bbE_V$ be the representation variety of $Q$ over an $I$-graded space $V$ and denote by $M(x)$ the representation $(V,x)$ corresponding to $x\in\bbE_V$.
For $(\bfd,m)\in\tcG(Q)$, let
\begin{eqnarray}
&X(\bfd,m)=\{x\in\bbE_V|M(x)\cong M(\bfd)\oplus Z_1\oplus\dots\oplus Z_m, \notag\\
&\text{for some simple objects}\, Z_1,\dots,Z_m\, \text{in different homogeneous regular tubes}\},\notag
\end{eqnarray}
$$Y(\bfd,m)=\{x\in\bbE_V|M(x)\cong M(\bfd)\oplus Z, \,\text{for some}\, Z\in\cR_{\rm h}, \udim Z=m\dz_Q\}.$$
Here $\dz_Q$ is the minimal imaginary root in $\bbN I$.
Obviously, we have $X(\bfd,m)\subset Y(\bfd,m)$.

By Proposition 4.14 in \cite{Lusztig_Affine_quivers_and_canonical_bases}, $X(\bfd,m)$ is a smooth, locally closed, irreducible $\GL_V$-invariant subvariety of $\bbE_V$. For $\lz\vdash m$, let $\cL_\lz$ be the irreducible local system on $X(\bfd,m)$ such that the monodromy representation at the stalk of $x\in X(\bfd,m)$ induced by $\cL_\lz$ is the Specht module $S_\lz$ (see \cite{Lusztig_Affine_quivers_and_canonical_bases} and \cite{Li_Notes_on_affine_canonical_and_monomial_bases}). 

For each $(\bfd,t_\lz)\in\cG^a(Q)$, let $\IC(\bfd,t_\lz)$ be the intersection cohomology sheaf on $\bbE_V$ whose support is the closure of  $X(\bfd,|\lz|)$ and whose restriction to $X(\bfd,|\lz|)$ is $\cL_\lz$ up to shift. Then any simple perverse sheaf in $\cP_V$ is isomorphic to $\IC(\bfd,t_\lz)$ for some $(\bfd,t_\lz)\in\cG^a(Q)$.

Assume $|V|\in(\bbZ I)^\sz$, and fix an automorphism $\sz:V\ra V$, then the induced functor $\sz^*:\cD(\bbE_V)\ra \cD(\bbE_V)$ maps $\IC(\bfd,t_\lz)$ to $\IC(\sz^{-1}(\bfd),t_\lz)$, where $\sz:\cG'(Q)\ra\cG'(Q)$ is a bijection such that $M(\sz(\bfd))=M(\bfd)^{[1]}$. Therefore any $[L,\phi]\in\cB$ must satisfy $L\cong \IC(\bfd,t_\lz)$ for some $(\bfd,t_\lz)\in\cG^a(Q)$ such that $\sz(\bfd)=\bfd$.

For each $(\bfd,t_\lz)$ such that $\sz(\bfd)=\bfd$, let $\tphi: \sz^*\IC(\bfd,t_\lz)\ra \IC(\bfd,t_\lz)$  be the unique isomorphism such that $(D(\IC(\bfd,t_\lz)),D(\tphi)^{-1})\cong (\IC(\bfd,t_\lz),\tphi)$ and $\epsilon([\IC(\bfd,t_\lz),\tphi])=1$. 
The isomorphism $\tphi$ is denoted by $\az_P$ for $P=\IC(\bfd,t_\lz)$ in \cite{Lusztig_Canonical_bases_and_Hall_algebras}, which is called the canonical isomorphism.
Then by Lusztig, we have the following lemma which gives a geometric interpretation of the affine canonical basis.

\begin{lemma}
In $K(\tcQ)$, let 
$$\bfB''=\{[\IC(\bfd,t_\lz),\tphi]|(\bfd,t_\lz)\in\cG^a(Q),\sz(\bfd)=\bfd\},$$ 
then $\Xi_2(\bfB)=\bfB''$.
\end{lemma}

Note that for each $\bfc\in\cG'(\ggz)$, by Section \ref{sec Frob fold}, there is a unique $\bfd\in\cG'(Q)$ such that
$\sz(\bfd)=\bfd$ and $M_k(\bfc)\in\Rep_k(\ggz)$ is isomorphic to the Frobenius folding of $M_{\bar{k}}(\bfd)\in\Rep_{\bar{k}}Q$. 
Denote by $\eta$ the injection $\bfc\mapsto \bfd$. 
Easy to check that $\eta$ maps aperiodic indices to aperiodic ones. 
By abuse of notation, we also denote by $\eta$ the $\bbZ$-linear map $\bbZ \bfI\ra\bbZ I, \bfi\mapsto\sum_{i\in\bfi}i$. 
Easy to see that $\udim M_{\bar{k}}(\eta(\bfc))=\eta(\udim M_k(\bfc))$ for $\bfc\in\cG'(\ggz)$.

Denote $B(\bfc,t_\lz)=[\IC(\eta(\bfc),t_\lz),\tphi]$. Then $$\bfB''=\{B(\bfc,t_\lz)|(\bfc,t_\lz)\in\cG^a(\ggz)\}$$

\subsection{Trace map}

Consider the $\bbQ$-algebra homomorphism $\Phi_k: K(\tcQ)\cong \bfU^+_\cA\cong \cC^*_\cA(\ggz)\hookrightarrow \cH^*(\ggz)\ra\cH^*_k(\ggz)$. 
By \cite{Lusztig_Canonical_bases_and_Hall_algebras}, for any $[L,\phi]\in\cB$, the image $\Phi_k([L,\phi])$ can be computed as follow.

Let $k=\bbF_q$ and the ground field be $\bar{k}$. 
Given a Frobenius map $F:V\ra V$ such that $F(V_i)=V_i$, then there is a naturally induced Frobenius map $F:\bbE_V\ra\bbE_V$ such that $F(x)_h(F(v))=F(x_h(v))$ for any $x\in\bbE_V, h\in H$ and $v\in V_{s(h)}$. 
Then for each simple $L'\in\cP_V$, we have an canonical isomorphism $\xi_{L'}:F^*L'\ra L'$ given by Theorem 5.2 in \cite{Lusztig_Canonical_bases_and_Hall_algebras}.

For  $[L,\phi]\in\cB$, consider the composition
 $$(\sz F)^*L=F^*\sz^*L\xrightarrow{F^*(\phi)} F^*L\xrightarrow{\xi_L}L.$$ 
This induces an automorphism $$H^j(L)|_x=H^j(L)|_{\sz F(x)}=H^j((\sz F)^*L)|_x\cong H^j(L)|_x$$
for all $j\in\bbZ$ and $x\in\bbE_V^{\sz F}$,
where $H^j(L)|_x$ is the stalk at $x$ of the $j$-th cohomology sheaf of $L$.
Taking the trace of this automorphism times $(-1)^j$ and summing over $j\in\bbZ$, we obtain an element denoted by $\chi_{L,\phi,k}(x)\in\overline{\bbQ_l}\cong \bbC$.

Note that each $x\in\bbE_V^{\sz F}$ induces a representation in $\Rep_k \ggz$. By identifying the characteristic function $\chi_{[M]}$ of $M\in \Rep_k \ggz$ with $\lan M\ran$ in $\cH^*_k(\ggz)$, we have $\chi_{L,\phi,k}=\Phi_k([L,\phi])$.

\section{The signed canonical basis}

\subsection{Main theorem}

Let $\bfB'=\{C(\bfc,t_\lz)|(\bfc,t_\lz)\in\cG^a(\ggz)\}$ be the bar-invariant basis of $\cC^*_{\cA'}(\ggz)$ constructed in Section \ref{sec bar-inv basis} and $\bfB$ be the affine canonical basis of $\bfU^+$ introduced in Section \ref{sec Simple perverse sheaves for tame quivers}.
Let $\cB=\bfB\sqcup{-\bfB}$ be the signed canonical basis and $\cB'=\bfB'\sqcup{-\bfB'}$.
Denote the isomorphism $\Xi=\Xi_2\circ(\Xi_1^{-1}):\cC^*_\cA\ra K(\tcQ)$.

\begin{theorem}\label{main theorem}
Under the isomorphism $\Xi:\cC^*_\cA\ra K(\tcQ)$, we have 
$$\Xi(\cB')=\cB''.$$
Moreover, the isomorphism $\Xi$ maps $\pm C(\bfc,t_\lz)$ to $\pm B(\bfc,t_\lz)$.
\end{theorem}

\subsection{Proof of Theorem \ref{main theorem}}

\begin{lemma}
For $(\bfc,t_\lz)\in\cG^a(\ggz)$, we have
$$L_{\oz(\bfc,t_\lz)}\cong\IC(\eta(\bfc),t_\lz)\oplus L',$$
where $\supp L'\subset \overline{X(\eta(\bfc),|\lz|)}\setminus X(\eta(\bfc),|\lz|)$.
\end{lemma}
\begin{proof}
Since $L_{\oz(\bfc,t_\lz)}\cong L_{\hoz(\bfc,t_\lz)}$, by \cite{XXZ_Tame_quivers_and_affine_bases_I}, it suffices to show that 
$$\fkm^{\hoz(\bfc,t_\lz)}=N_Q(\eta(\bfc),t_\lz)+\sum_{(\bfd,t_\mu)\prec(\eta(\bfc),t_\lz)}\phi^{(\bfd,t_\mu)}_{(\bfc,t_\lz)} N_Q(\bfd,t_\mu)$$
holds in $\cH^0(Q)$ for some $\phi^{(\bfd,t_\mu)}_{(\bfc,t_\lz)}\in\cA$. 
Here we use $N_Q(-)$ instead of $N(-)$ to emphasize that we are now dealing in $\cH^0(Q)$ instead of  $\cH^0(\ggz)$, where $Q$ is a (non-valued) quiver and $\ggz=\ggz(Q,\sz)$.

By the proof of Proposition 9.10 in \cite{XXZ_Tame_quivers_and_affine_bases_I}, we need only to prove the following equations:
$$\fkm^{\hoz(\bfc_-)}=N_Q(\eta(\bfc_-),0)+\sum_{(\bfd,t_\mu)\in\cG(Q)\atop\bfd_->_L\bfc_-}\phi_{\bfc_-}^{(\bfd,t_\mu)}N_Q(\bfd,t_\mu),$$
$$\fkm^{\hoz(\bfc_+)}=N_Q(\eta(\bfc_+),0)+\sum_{(\bfd,t_\mu)\in\cG(Q)\atop\bfd_+>_L\bfc_+}\phi_{\bfc_+}^{(\bfd,t_\mu)}N_Q(\bfd,t_\mu),$$
$$\circledast: \fkm^{\hoz(\bfc_0)}=\lan M(\eta(\bfc_0))\ran+\sum_{\bfd_0\in\cG_0(Q)\atop\bfd_0<_G\eta(\bfc_0)}\phi^{\bfd_0}_{\bfc_0}(v)\lan M(\bfd_0)\ran+\sum_{(\bfd,t_\mu)\in\cG(Q)\atop\bfd_->0,\bfd_+>0}\phi^{(\bfd,t_\mu)}_{\bfc_0}(v)N_Q(\bfd,t_\mu),$$
$$\fkm^{\hoz(\lz\dz)}=H_\lz+\sum_{(\bfd,t_\mu)\in\cG\atop\bfd_\pm=0,\bfd_0\neq 0}\phi^{(\bfd,t_\mu)}_{\lz\dz}(v)N_Q(\bfd,t_\mu)+\sum_{(\bfd,t_\mu)\in\cG\atop\bfd_->0,\bfd_+>0}\phi^{(\bfd,t_\mu)}_{\lz\dz}(v)N_Q(\bfd,t_\mu).$$
Here all $\phi^-_-\in\cA$.

We prove $\circledast$ here, while the rest ones are routine to check by Lemma 9.2, Lemma 9.3, Lemma 9.8 in \cite{XXZ_Tame_quivers_and_affine_bases_I} and definition of these monomials.

For a non-homogeneous tube $\cT$ in $\Rep_k \ggz$ and a regular simple $L$ in $\cT$, by Lemma \ref{how nh tubes are F-folded}, there are two cases to be considered.

For case $(a)$ in Lemma \ref{how nh tubes are F-folded}, $L\cong (L_0\oplus L_1 \oplus\dots\oplus L_{l-1})^F$, then
$$\eta(\udim L)=\sum_{0\leq t\leq l-1}\udim L_t.$$
By comparing dimension vectors, for $a\geq 1$ we have
$$\fkm^{a\eta(\udim L)}\in\lan L_0\oplus L_1 \oplus\dots\oplus L_{l-1}\ran+\sum_{\bfd_\pm\neq 0}\cA'N_Q(\bfd,t_\mu),$$
then by construction, for any $\pi\in\Pi^a$, we have
$$\fkm^{\hoz(\pi)}\in\lan M_{T}(\pi)\ran*\lan M_{T^{[1]}}(\pi)\ran*\dots*\lan M_{T^{[l-1]}}(\pi)\ran+\sum_{\bfd_0<_G\bfc^\pi}\cA'\lan M(\bfd_0)\ran+\sum_{\bfd_\pm\neq 0}\cA'N_Q(\bfd,t_\mu),$$
here $\bfc^\pi$ is the index such that $\bfc^\pi_\pm=0$ and takes $\pi$ on tube $T,T^{[1]},T^{[l-1]}$ and $0$ on others.

For case $(b)$ in Lemma \ref{how nh tubes are F-folded}, $L\cong (L_1\oplus L_2)^F$, we also have
$$\fkm^{a\eta(\udim L)}\in\lan L_1\oplus L_2\ran+\sum_{\bfd_\pm\neq 0}\cA'N_Q(\bfd,t_\mu).$$
Let $r=\rk \cT\geq 2$. By abuse of notation, denote by $\eta:\Pi^a_r\ra\Pi^a_{2r}$ such that $\eta([i;l))=[i;l)+[i+r;l)$ and $\eta$ is spanned linearly.
To obtain the similar conclusion, we need only to prove that $\eta(\pi)\diamond\eta(\pi')=\eta(\pi\diamond\pi')$ for any $\pi,\pi'\in\Pi^a_r$. 

We check this for $\pi=[j;1)$. Let $l_0$ be maximal such that $\pi'_{j+1,l_0}\leq 0$. By \cite{Deng_Du_Xiao_Generic_extensions_and_canonical_bases_for_cyclic_quivers}, we have
\begin{eqnarray*}
\eta([j;1)\diamond\pi')&=&\eta(\pi'-[j+1;l_0)+[j;l_0+1))\\
&=&\eta(\pi')-[j+1;l_0)-[j+i+r;l_0)+[j;l_0+1)+[j+r;l_0+1),
\end{eqnarray*}
\begin{eqnarray*}
\eta([j;1))\diamond\eta(\pi')&=&[j;1)\diamond[j+r;1)\diamond\eta(\pi')\\
&=&[j;1)\diamond(\eta(\pi')-[j+r+1;l_0)+[j+r;l_0+1))\notag\\
&=&\eta(\pi')-[j+r+1;l_0)+[j+r;l_0+1)-[j+1;l_0)+[j;l_0+1),\notag
\end{eqnarray*}
then $\eta([j;1))\diamond\eta(\pi')=\eta([j;1)\diamond\pi')$. 
By induction on length, let $\pi=[j;1)\diamond \pi''$ for some $j$ and $\pi''$, then 
$$\eta(\pi)\diamond\eta(\pi')=\eta([i;1))\diamond\eta(\pi'')\diamond\eta(\pi')=\eta([i;1))\diamond\eta(\pi''\diamond\pi')=\eta([i;1)\diamond\pi''\diamond\pi')=\eta(\pi\diamond\pi'').$$

Therefore, in both cases, we have 
$$\fkm^{\hoz(\pi)}\in N(\eta(\pi))+\sum_{\bfd_0<_G\eta(\pi)}\cA' N(\bfd_0,0)+\sum_{\bfd_\pm\neq 0}\cA'N_Q(\bfd,t_\mu),$$
By multiplying these for $\bfc_0=(\pi_1,\pi_2,\dots,\pi_s)$, $\circledast$ is implied.

\end{proof}

By applying $(-,\phi_0)$ to the both sides, where $\phi_0:\sz^*L_{\oz(\bfc,t_\lz)}\ra L_{\oz(\bfc,t_\lz)}$ is an isomorphism (see Section \ref{section: Perverse sheaves for valued quivers}), we have
$$(L_{\oz(\bfc,t_\lz)},\phi_0)\cong(\IC(\eta(\bfc),t_\lz)\oplus L',\phi_0)\cong(\IC(\eta(\bfc),t_\lz),\phi_0|_{\IC(\eta(\bfc),t_\lz)})\oplus(L',\phi_0|_{L'}).$$
We abbreviate $\tilde{\phi}_0=\phi_0|_{\IC(\eta(\bfc),t_\lz)}$.

By applying $\Xi$, in $\cK(\tcQ)$ we have
$$\fkm^{\oz(\bfc,t_\lz)}=[\IC(\eta(\bfc),t_\lz),\tilde{\phi}_0]+[L',\phi_0|_{L'}].$$
By Lemma \ref{lemma: CB is unique up to sign}, $$[\IC(\eta(\bfc),t_\lz),\tilde{\phi}_0]=\pm B(\bfc,t_\lz)$$ and   $$[L',\phi_0|_{L'}]\in\sum_{(\bfc',t_{\lz'})\prec(\bfc,t_\lz)}\cA B(\bfc',t_{\lz'}).$$

This shows that the monomial basis $\{\fkm^{\oz(\bfc,t_\lz)}|(\bfc,t_\lz)\in\cG^a\}$ is upper-triangular to $\bfB''$ up to $\pm$.
Then by Proposition \ref{monomial basis},
$$\Xi^{-1}(B(\bfc,t_{\lz}))=\pm N(\bfc,t_\lz)+\sum_{(\bfc',t_{\lz'})\in\cG\atop(\bfc',t_{\lz'})\prec(\bfc,t_\lz)}\zeta_{(\bfc,t_\lz)}^{(\bfc',t_{\lz'})}(v)N(\bfc',t_{\lz'})$$
for some $\zeta_-^-\in\cA'$.
Since both $\bfB''$ and $\{N(\bfc,t_\lz)|(\bfc,t_\lz)\in\cG\}$ are almost orthogonal with respect to the bilinear from $(-,-)$, an elementary argument shows that all $\zeta^-_-\in v^{-1}\bbQ[v^{-1}].$
Then
$$\Xi^{-1}(B(\bfc,t_{\lz}))-\sum_{(\bfc',t_{\lz'})\in\cG^a\atop(\bfc',t_{\lz'})\prec(\bfc,t_\lz)}\zeta_{(\bfc,t_\lz)}^{(\bfc',t_{\lz'})}(v)E(\bfc',t_{\lz'})\in
\pm N(\bfc,t_\lz)+\sum_{(\bfc'',t_{\lz''})\in\cG\setminus\cG^a\atop(\bfc'',t_{\lz''})\prec(\bfc,t_\lz)}\cA'N(\bfc'',t_{\lz''}).$$
By Proposition \ref{uniqueness of PBW}, we have
$$\Xi^{-1}(B(\bfc,t_{\lz}))=\pm E(\bfc,t_\lz)+\sum_{(\bfc',t_{\lz'})\in\cG^a\atop(\bfc',t_{\lz'})\prec(\bfc,t_\lz)}\zeta_{(\bfc,t_\lz)}^{(\bfc',t_{\lz'})}(v)E(\bfc',t_{\lz'})$$
with $\zeta^-_-\in v^{-1}\bbQ[v^{-1}].$

Through Proposition $\ref{relation bt C and E}$, we know that the transition matrix between $\Xi^{-1}(\bfB'')$ and $\bfB'$ is upper-triangular with entries in the main diagonal are $\pm 1$ and other entries are in $v^{-1}\bbQ[v^{-1}]$. Since these two bases are both bar-invariant, all entries except those in the main diagonal shall be 0. 
Therefore we have
$$\cB'=\bfB'\sqcup(-\bfB')=\Xi^{-1}(\bfB'')\sqcup(-\Xi^{-1}(\bfB''))=\Xi^{-1}(\cB'').$$

\section{Kashiwara's operators}\label{Sec: Kashi}

\subsection{Kashiwara's operators}

We recall the setup for Kashiwara's operators given in [\cite{Lusztig_Introduction_to_quantum_groups}, Chapter 16]. 

We fix a $\bbQ(v)$-vector space $P$ and $\bfi\in\bfI$, and suppose that there are two $\bbQ(v)$-linear maps $\ez,\phi:P\ra P$ such that $\ez$ is locally nilpotent and $\ez\phi=v_\bfi^2\phi\ez+1$.
We call $(P,\ez,\phi)$ an admissible triple (of $\bfi$).

Let $\phi^{(s)}:P\ra P$ be defined as $\phi^{(s)}/[s]_\bfi^!$. Then by induction we have
$$\ez\phi^{(N)}=v_\bfi^{2N}\phi^{(N)}\ez+v_\bfi^{N-1}\phi^{(N-1)}$$
for all $N$. Define for $N\geq 0$ the subspaces $P(N)=P_{\ez,\phi}(N)$ of $P$ by
$$P(0)=\{x\in P|\ez(x)=0\}, P(N)=\phi^{(N)}P(0).$$
Then $P=\oplus_{N\geq 0}P(N)$.

Define the associated Kashiwara's operators $\tez,\tphi: P\ra P$ by
$$\tphi(\phi^{(N)}y)=\phi^{(N+1)}y,\tez(\phi^{(N)}y)=\phi^{(N-1)}y$$
for all $y\in P(0)$.

Now consider $P=\bbQ(v)\otimes_{\cA'}\cH^0(\ggz)$. 
We shall remind the readers that this is different from [\cite{Lusztig_Introduction_to_quantum_groups}, Chapter 17] where $P$ is chosen to be $\cC^*$.
For each $\bfi\in\bfI$, let
$\phi_\bfi:P\ra P$ acts as left multiplication by $u_\bfi$ and $\ez_\bfi:P\ra P$ acts as ${}_\bfi r$ (See Section \ref{sec: Almost orthogonality}). 
One can check that $(P,\ez_\bfi,\phi_\bfi)$ is an admissible triple, thus $\tez_\bfi,\tphi_\bfi$ are well defined.

Let $$\cL_\cA=\{x\in \cC^*_\cA(\ggz)\subset P|(x,x)\in\bbQ[[v^{-1}]]\cap\bbQ(v)\}$$ and
$$\cL_{\cA'}=\{x\in \cH^0(\ggz)\subset P|(x,x)\in\bbQ[[v^{-1}]]\cap\bbQ(v)\}$$
be the subsets of $P$. 
Obviously we have $\cL_\cA\subset\cL_{\cA'}$.
Moreover, the almost orthogonality of $\{N(\bfc,t_\lz)|(\bfc,t_\lz)\in\cG\}$ shows that 
$$\cL_{\cA'}=\{\sum_{(\bfc,t_\lz)\in\cG}a(\bfc,t_\lz)N(\bfc,t_\lz)|a(\bfc,t_\lz)\in\bbQ[[v^{-1}]]\cap\bbQ(v)\}.$$

By construction of bases, we have the following lemma:
\begin{lemma}\label{lemma: C=E=N mod v-1}
    For $(\bfc,t_\lz)\in\cG^a$, we have
    $$C(\bfc,t_\lz)\equiv E(\bfc,t_\lz) \equiv N(\bfc,t_\lz)\,\mod v^{-1}\cL_{\cA'}.$$
\end{lemma}

The next lemma is similar to Lemma 16.2.8 in \cite{Lusztig_Introduction_to_quantum_groups}.
\begin{lemma}\label{lemma: Kashiwara's operators keep LA'}
    For all $\bfi\in\bfI$, $\tez_\bfi(\cL_{\cA'})\subset\cL_{\cA'},\tphi_\bfi(\cL_{\cA'})\subset\cL_{\cA'}$.
\end{lemma}
\begin{proof}
    See [\cite{Lusztig_Introduction_to_quantum_groups} Lemma 16.2.6], we can prove similarly that for $x\in P(0)$,
    $$(\phi_\bfi^{(N+1)} x,\phi_\bfi^{(N+1)} x)/(\phi_\bfi^{(N)} x,\phi_\bfi^{(N)} x)\in 1+v^{-1}\bbZ[v^{-1}].$$
    For any $y\in P$, write uniquely $y=\sum_{N\geq 0} \phi_\bfi^{(N)}y_N$ with $y_N\in P(0)$. 
    We show that $y\in\cL_{\cA'}$ if and only if $y_N\in\cL_{\cA'}$ for all $N$. 
    
    The if-part is easy since $(y,y)=\sum_{N}(\phi_\bfi^{(N)} y_N,\phi_\bfi^{(N)} y_N)$.
    Suppose there is some $N$ such that $y_N\notin\cL_{\cA'}$, then $(y_N,y_N)\in v^{2d}(\bbQ^+ +v^{-1}\bbQ[[v^{-1}]]\cap\bbQ(v))$ for some $d>0$. Let $d_m$ be the maximal $d$ among all $N$.
    Therefore the coefficient of $v^{2d_m}$ in $(y,y)=\sum_{N}(\phi_\bfi^{(N)} y_N,\phi_\bfi^{(N)} y_N)$ is nonzero. Thus the claim is right.

    Now let $y\in\cL_{\cA'}$, then $y_N\in\cL_{\cA'}$ for all $N$. We deduce that 
    $$\tphi_\bfi(y)=\sum_{N\geq 0}\phi^{(N+1)}y_N$$
    and
    $$\tez_\bfi(y)=\sum_{N>0}\phi^{(N-1)}y_N$$
    are both in $\cL_{\cA'}$.
    
\end{proof}

 Consider the subset $\cB_{\bfi,n}$ of $\cB$ and the bijection $\varphi_{\bfi;n}:\cB_{\bfi;0}\ra\cB_{\bfi;n}$ defined in Section \ref{sec: canonical basis}.  
 Denote $\cB'_{\bfi,n}=\Xi_1(\cB_{\bfi,n})\subset\cC^*\subset P$ and (by abuse of notation) $\varphi_{\bfi;n}:\cB'_{\bfi;0}\ra\cB'_{\bfi;n}$.
 Note that when restricted to $\cC^*$, $\ez_\bfi|_{\cC^*},\phi_\bfi|_{\cC^*}$ are the operators introduced in [\cite{Lusztig_Introduction_to_quantum_groups}, chp 17.3]. 

\begin{lemma}\label{Kashiwara's operator acts like pi}
    For $\bfi\in\bfI,n\geq 0$ and $b\in\cB'_{\bfi;0}$ we have
   $$\tphi^n_\bfi(b)\equiv \varphi_{\bfi;n}(b)\,\mod v^{-1}\cL_{\cA}.$$
\end{lemma}

\begin{proof}
    This follows directly from [\cite{Lusztig_Introduction_to_quantum_groups}, Corollary 17.3.7].
\end{proof}

\begin{remark}\label{rmk: Kashiwara's operator acts like pi}
    We can replace $\cL_\cA$ by $\cL_{\cA'}$ in the previous lemma and it is still true since $\cL_\cA\subset\cL_{\cA'}$.
\end{remark}

Denote by $\varphi_\bfi:\cB'\ra\cB'$ such that $\varphi_\bfi$ acts on $\cB'_{\bfi;n}$ by $\varphi_{\bfi;n+1}\circ\varphi_{\bfi;n}^{-1}$.

\begin{corollary}\label{cor: Kashiwara's operator acts like pi}
     For $\bfi\in\bfI$ and $b\in\cB'$ we have
   $$\tphi_\bfi(b)\equiv \varphi_{\bfi}(b)\,\mod v^{-1}\cL_{\cA}.$$
\end{corollary}
\begin{proof}
    Let $b\in\cB'_{\bfi;n}$ and $b_0=\varphi_{\bfi;n}^{-1}(b)\in\cB'_{\bfi;0}$. Then
    $$\varphi_{\bfi}(b)=\varphi_{\bfi;n+1}(b_0)\equiv \tphi^{n+1}_\bfi(b)\equiv\tphi_\bfi(\varphi_{\bfi;n}(b_0))=\tphi_\bfi(b)\,\mod v^{-1}\cL_{\cA}$$
    by Lemma 16.2.8 in \cite{Lusztig_Introduction_to_quantum_groups} and Lemma \ref{Kashiwara's operator acts like pi}.
\end{proof}

For $\bfc\in\cG'$, let $\bfc^{\vdash}\in\cG'$ be such that $\bfc^{\vdash}_{+}=\bfc_{+},\bfc^{\vdash}_{0}=\bfc_{0}$,  $\bfc^{\vdash}_{-}(t)=\bfc_{-}(t)$ for $t<0$ and $\bfc^{\vdash}_{-}(0)=0$.

\begin{proposition}\label{prop: pi acts on the left-most side}
    For $\bmbz_0=\bfi_0$ and any $(\bfc,t_\lz)\in\cG^a$, let $N=\bfc_-(0)$, then we have $C(\bfc,t_\lz)\in\cB'_{\bfi_0;N}$. In particular, $C(\bfc^{\vdash},t_\lz)\in\cB'_{\bfi_0;0}$. 
    Moreover, we have
    $$\varphi_{\bfi_0,N}(C(\bfc^{\vdash},t_\lz))=C(\bfc,t_\lz).$$
\end{proposition}

\begin{proof}
    We already know that $C(\bfc,t_\lz)\in\cB'$, so we need to show that $C(\bfc,t_\lz)\in u_{\bfi_0}^N\cC^*$ and $C(\bfc,t_\lz)\notin u_{\bfi_0}^{N+1}\cC^*$.
    
    By Proposition \ref{PBW basis is multiplicative}, 
    $$E(\bfc',t_{\lz'})=u_{\bfi_0}^{(\bfc'_-(0))}*E((\bfc')^{\vdash},t_{\lz'})\in u_{\bfi_0}^{\bfc'_-(0)}\cC^*$$
    for all $(\bfc',t_{\lz'})\in\cG^a$.
    Since $C(\bfc,t_\lz)$ is an $\cA'$-sum of some $E(\bfc',t_{\lz'})$ with $(\bfc',t_{\lz'})\preceq (\bfc,t_\lz)$, by the lexicographic ordering on $\cG_-$, $\bfc'_-(0)\geq N$, so $C(\bfc,t_\lz)\in u_{\bfi_0}^{N}\cC^*$.

    On the other side, assume that $C(\bfc,t_\lz)\in u_{\bfi_0}^{N+1}\cC^*$, then $C(\bfc,t_\lz)=u_{\bfi_0}^{(N+1)}x$ for some $x\in\cC^*$. 
    Thus by Proposition \ref{multiplication of N}, $C(\bfc,t_\lz)$ is an $\cA'$-sum of $E(\bfc',t_{\lz'})$ with $\bfc'_-(0)>N$. This contradicts to the fact that $E(\bfc,t_{\lz})$ is the leading term of $C(\bfc,t_\lz)$.

    By Lemma \ref{lemma: C=E=N mod v-1}, Lemma \ref{lemma: Kashiwara's operators keep LA'}, Remark \ref{rmk: Kashiwara's operator acts like pi}, we have 
    $$\varphi_{\bfi_0,N}(C(\bfc^{\vdash},t_\lz))\equiv \tphi^N_{\bfi_0}(C(\bfc^{\vdash},t_\lz))\equiv \tphi^N_{\bfi_0}(N(\bfc^{\vdash},t_\lz))\,\mod v^{-1}\cL_{\cA'}.$$
    
    Since $\bfi_0$ is a sink of $\ggz$, and any module $L$ appearing in the $k$-specialization of $N(\bfc^{\vdash},t_\lz)$ has no direct summands isomorphic to the simple module $S_{\bfi_0}$ at $\bfi_0$, the Hall number $g^L_{S_\bfi M}=0$ for any $M\in\Rep_k\ggz$. Therefore 
    $$\ez_{\bfi_0}(N(\bfc^{\vdash},t_\lz))={}_{\bfi_0}r(N(\bfc^{\vdash},t_\lz))=0$$
    by definition of $r$ (See Section \ref{sec: Almost orthogonality}). That is, $N(\bfc^{\vdash},t_\lz)\in P(0)$.
    
    Then 
    $$\tphi^N_{\bfi_0}(N(\bfc^{\vdash},t_\lz))=\phi_{\bfi_0}^{(N)}(N(\bfc^{\vdash},t_\lz))=u_{\bfi_0}^{(N)}*N(\bfc^{\vdash},t_\lz)=N(\bfc,t_\lz).$$
    Finally, we have
    $$\varphi_{\bfi_0,N}(C(\bfc^{\vdash},t_\lz))\equiv N(\bfc,t_\lz)\,\mod v^{-1}\cL_{\cA'}.$$
    Since $C(\bfc,t_\lz)$ is the only element in $\cB'$ such that 
    $$C(\bfc,t_\lz)\equiv N(\bfc,t_\lz)\,\mod v^{-1}\cL_{\cA'},$$
    we have $$\varphi_{\bfi_0,N}(C(\bfc^{\vdash},t_\lz))=C(\bfc,t_\lz).$$
\end{proof}

\subsection{Kashiwara's operators for a simple in $\cR^\ggz$}

Assume in this section that $\bfi\in\bfI$ is a vertex such that $S_\bfi$ is a regular (hence non-homogeneous regular) module.

Without loss of generality, we assume that $S_\bfi$ is in the nonhomogeneous tube $\cT=\cT_1$. For convenience, we denote $\cK=\cK_{r_1}(k_1)$ and $\vez: \cK\ra\cT$ be the equivalence. Denote $\Pi=\Pi_{r_1}$. 

For a vertex $j$ in the cyclic quiver $K=K_{r_1}$, the operators $\ez_j={}_jr$ and $\phi_j=u_j*(-)$ are well defined for $P^K=\bbQ(v)\otimes_{\cA'}\cH^0(K)=\oplus_{\pi\in\Pi}\bbQ(v)\lan M(\pi)\ran$, then $(P^K,\ez_j,\phi_j)$ is an admissible triple of $j$, and we have the associated Kashiwara's operators $\tez_j,\tphi_j$. Denote 
$$\cL^K_{\cA'}=\{x\in\cH^0(K)\subset P|(x,x)\in\bbQ[[v^{-1}]]\cap\bbQ(v)\},$$
then
$$\cL^K_{\cA'}=\{\sum_{\pi\in\Pi}a(\pi)\lan M(\pi)\ran|a(\pi)\in\bbQ[v^{-1}]\}.$$
By \cite{Deng_Du_Xiao_Generic_extensions_and_canonical_bases_for_cyclic_quivers}, let $\bfB^K=\{B^K(\pi)|\pi\in\Pi^a\}$ be the canonical basis of $\cC^*(K)$ such that $$B^K(\pi)\equiv \lan M(\pi)\ran \,\mod v^{-1}\cL^K_{\cA'}, \pi\in\Pi^a.$$
Let $\cB^K=\bfB^K\sqcup(-\bfB^K)$.
Then the subsets $\cB^K_{j;n}$, the bijections $\varphi^K_{j;n}:\cB^K_{j;0}\ra\cB^K_{j;n}$ and maps $\varphi^K_{j}:\cB^K\ra\cB^K$ are well defined.

For any $\pi\in\Pi^a$, denote by
$\pi\uparrow_j\in\Pi^a$ such that 
$$\varphi^K_{j}(B^K(\pi))=B^K(\pi\uparrow_j).$$

\begin{proposition}\label{C(pi)=phi C(pi|)}
    We fix $(\pi_2,\dots,\pi_s)\in \Pi^a_{r_2}\times \dots\times\Pi^a_{r_s}$ and $t_\lz\in\bbP$. 
    Let $i$ be the vertex in the quiver $K$ such that $\vez(S_i)=S_\bfi$. 
    Then
    $$\varphi_{\bfi}C((\pi,\pi_2,\dots,\pi_s),t_\lz)=C((\pi\uparrow_i,\pi_2,\dots,\pi_s),t_\lz).$$
\end{proposition}
  
\begin{proof}
    By Lemma \ref{lemma: Kashiwara's operators keep LA'} and Corollary \ref{cor: Kashiwara's operator acts like pi}, we have 
    $$(*):\,\lan M(\pi\uparrow_i) \ran\equiv B^K(\pi\uparrow_i)=\varphi^K_{i}(B^K(\pi))\equiv\tphi_i(B^K(\pi))\equiv \tphi_i(\lan M(\pi) \ran)\,\mod v^{-1}\cL^K_{\cA'}.$$
    Note that the induced isomorphism $\vez:\cH^0(K)\ra\cH^0(\cT)$ satisfies
    $$\vez\circ\phi_i=\vez(u_i*(-)))=u_\bfi*\vez(-)=\phi_\bfi\circ\vez.$$
    
    We claim that $\vez({}_ir(x))={}_\bfi r(\vez(x))$ for any $x\in\cH^0(K)$, hence $\vez\circ\ez_i=\ez_\bfi\circ\vez$.
    We need only to show this for $x=\lan J\ran,J\in\cK$.
    Let $\vez(J)=L\in\cT$,
    then in the $k$-specialization,
    $${}_\bfi r(\lan L\ran)=\sum_{[M],M\in\Rep_k\ggz}v_k^{\lan \bfi,\udim M\ran+\dim_k\End L-\dim_k\End S_\bfi-\dim_k\End M}g^L_{S_\bfi M}\frac{a_{S_{\bfi}}a_M}{a_L}\lan M\ran.$$
    Remind that $S_\bfi\in\cT$, then $g^L_{S_\bfi M}\neq 0$ implies $M\in\cT$. 
    Applying $\vez^{-1}$ to both sides will prove our claim.

    By [\cite{Lusztig_Introduction_to_quantum_groups}, 16.1], $\tphi$ is local-finitely generated by $\ez,\phi$ as operators on $P$ for any admissible triple $(P,\ez,\phi)$, then $\vez\circ\tphi_i=\tphi_\bfi\circ\vez$. 
    Apply $\vez$ to $(*)$, we get
    $$(**):\, N((\pi\uparrow_i,0,\dots,0),0) \equiv\tphi_\bfi(N((\pi,0,\dots,0),0))\,\mod v^{-1}\vez(\cL^K_{\cA'}).$$

    Note that for any $x\in\cH^0(\cT)$ and  $y\in\cH^0(\cT_2)\otimes\dots\otimes\cH^0(\cT_s)\otimes\cA'[H_1,H_2,\dots]$, 
    ${}_\bfi r(y)=0$, then 
    $${}_\bfi r(x*y)={}_\bfi r(x)*y.$$
    Therefore $\tphi_\bfi(x*y)=\tphi_\bfi(x)*y$ for the same reason.

    Apply right-multiplication by $N((0,\pi_2,\dots,\pi_s),t_\lz)$ to $(**)$, we get
    $$N((\pi\uparrow_i,\pi_2,\dots,\pi_s),t_\lz) \equiv\tphi_\bfi(N((\pi_i,\pi_2,\dots,\pi_s),t_\lz))\,\mod v^{-1}\cL_{\cA'}$$
    since $\vez(\cL^K_{\cA'})*N((0,\pi_2,\dots,\pi_s),t_\lz)\subset \cL_{\cA'}$.

    By Lemma \ref{lemma: C=E=N mod v-1}, Lemma \ref{lemma: Kashiwara's operators keep LA'}, Lemma \ref{Kashiwara's operator acts like pi}, we have 
    \begin{align*}
    \varphi_{\bfi}C((\pi,\pi_2,\dots,\pi_s),t_\lz)&\equiv \tphi_\bfi C((\pi,\pi_2,\dots,\pi_s),t_\lz)\\
    &\equiv \tphi_\bfi(N((\pi,\pi_2,\dots,\pi_s),t_\lz))\\
    &\equiv N((\pi\uparrow_i,\pi_2,\dots,\pi_s),t_\lz)\,\mod v^{-1}\cL_{\cA'}.  
    \end{align*}
    Since $C((\pi\uparrow_i,\pi_2,\dots,\pi_s),t_\lz)\in\cB$ is the only element such that $$C((\pi\uparrow_i,\pi_2,\dots,\pi_s),t_\lz)\equiv N((\pi\uparrow_i,\pi_2,\dots,\pi_s),t_\lz)\,\mod v^{-1}\cL_{\cA'},$$
    the proposition follows.
\end{proof}

\subsection{Braid group action}

Let $T_\bfi=T''_{\bfi,1},T^{-1}_\bfi=T'_{\bfi,-1}:\bfU\ra\bfU$ be the Lusztig's symmetries defined in [\cite{Lusztig_Introduction_to_quantum_groups}, Chapter 37]. 
For $\bfi\in\bfI$, let
$$\bfU^+[\bfi]=\{x\in\bfU^+|T_\bfi(x)\in\bfU^+\}=\Ker {}_\bfi r,$$
$${}^\sz\bfU^+[\bfi]=\{x\in\bfU^+|T^{-1}_\bfi(x)\in\bfU^+\}=\Ker r_\bfi$$
by [\cite{Lusztig_Introduction_to_quantum_groups}, Proposition 38.1.6]. 
Then 
$T_\bfi:\bfU^+[\bfi]\ra {}^\sz\bfU^+[\bfi]$ is an isomorphism of subalgebras.

Also, we have $\bfU^+=\bfU^+[\bfi]\oplus E_\bfi\bfU^+={}^\sz\bfU^+[\bfi]\oplus\bfU^+E_\bfi$ as a vector space. 
Let ${}^\bfi\pi:\bfU^+\ra\bfU^+[\bfi]$ and $\pi^\bfi:\bfU^+\ra{}^\sz\bfU^+[\bfi]$ be the projections.

By identifying $\bfU^+$ with $\cC^*$, the subalgebras $\cC^*[\bfi],{}^\sz\cC^*[\bfi]$, the isomorphism $T_\bfi:\cC^*[\bfi]\ra {}^\sz\cC^*[\bfi]$ and the projections ${}^\bfi\pi:\cC^*\ra\cC^*[\bfi]$ and $\pi^\bfi:\cC^*\ra{}^\sz\cC^*[\bfi]$ are well-defined by abuse of notation.

\begin{lemma}[\cite{Lusztig_Braid_group_action_and_canonical_bases}, Proposition 1.8]\label{b-pi(b)}
    For $b\in\cB'-(\cB'\cap u_\bfi\cC^*)$, $b-{}^\bfi\pi(b)\in v^{-1}\cL_\cA$;
    For $b\in\cB'-(\cB'\cap \cC^* u_\bfi)$, $b-\pi^\bfi(b)\in v^{-1}\cL_\cA$.
\end{lemma}

\begin{lemma}[\cite{Lusztig_Braid_group_action_and_canonical_bases}, Theorem 1.2]\label{kappa is T}
    There is a unique bijection $\kappa_\bfi:\cB'-(\cB'\cap u_\bfi\cC^*)\ra \cB'-(\cB'\cap \cC^*u_\bfi)$ such that $T_\bfi({}^\bfi\pi(b))=\pi^\bfi(\kappa_\bfi(b))$ and $\sgn(\kappa_\bfi (b))=\sgn(b)$.
\end{lemma}

If $\bfi$ is a sink or source,
we have the Bernstein–Gelfand–Ponomarev reflection functors $\sz_\bfi^\pm:\Rep_k \ggz\ra\Rep_k\sz_\bfi\ggz$, which induce an exact equivalence $\sz_\bfi^\pm:\Rep_k \ggz\lan\bfi\ran\ra\Rep_k\sz_\bfi\ggz\lan \bfi\ran$, where $\Rep_k \ggz\lan\bfi\ran$ (resp., $\Rep_k\sz_\bfi\ggz\lan \bfi\ran$) is the full subcategory of $\Rep_k \ggz$ (resp., $\Rep_k\sz_\bfi\ggz$) of all representations which do not contain a direct summand isomorphic to the simple $S_\bfi$. 
Then $\sz_\bfi^\pm$ induce an isomorphism $\sz_\bfi^\pm: \cH^*_k(\ggz)\lan\bfi\ran\ra\cH^*_k(\sz_\bfi\ggz)\lan \bfi\ran$,  where $\cH^*_k (\ggz)\lan\bfi\ran$ (resp., $\cH^*_k(\sz_\bfi\ggz)\lan \bfi\ran$) is the subalgebra of $\cH^*_k (\ggz)$ (resp., $\cH^*_k(\sz_\bfi\ggz)$) spanned by isoclasses of modules in $\Rep_k \ggz\lan\bfi\ran$ (resp., $\Rep_k\sz_\bfi\ggz\lan \bfi\ran$).

By \cite{Deng_Xiao_Ringel-Hall_algebras_and_Lusztig's_symmetries}, Lustig's symmetries coincide with the BGP reflection functors up to a canonical isomorphism. 
Namely, $\sz_\bfi^\pm: \cH^*_k(\ggz)\lan\bfi\ran\ra\cH^*_k(\sz_\bfi\ggz)\lan \bfi\ran$ induces an isomorphism $\sz_\bfi^\pm:\cC^*(\ggz)\lan\bfi\ran\ra\cC^*(\sz_\bfi\ggz)\lan\bfi\ran$, where $\cC^*(\ggz)\lan\bfi\ran=\cC^*\cap \Pi_k\cH^*_k(\ggz)\lan\bfi\ran$. Moreover, under the isomorphism $\cC^*(\ggz)\cong\bfU^+\cong\cC^*(\sz_{\bfi}\ggz)$, $\sz_\bfi^+:\cC^*(\ggz)\lan\bfi\ran\ra\cC^*(\sz_\bfi\ggz)\lan\bfi\ran$ coincides with $T_\bfi:\bfU^+[\bfi]\ra {}^\sz\bfU^+[\bfi]$ when $\bfi$ is a sink and $\sz_\bfi^-:\cC^*(\ggz)\lan\bfi\ran\ra\cC^*(\sz_\bfi\ggz)\lan\bfi\ran$ coincides with $T_\bfi:{}^\sz\bfU^+[\bfi]\ra \bfU^+[\bfi]$.

By construction, we have
$$M(\bmbz_t)=\left\{
\begin{aligned}
&\sz^{-}_{\bfi_0}\sz^{-}_{\bfi_{-1}}\dots \sz^{-}_{\bfi_{t+1}}(S_{\bfi_t}),   & \text{if}\, t\leq 0,\\
&\sz^+_{\bfi_1}\sz^+_{\bfi_{2}}\dots \sz^+_{\bfi_{t-1}}(S_{\bfi_t}),    & \text{if}\, t>0.
\end{aligned}
\right.
$$
Then 
\begin{align*}
    E(\bfc_-)&=\lan M(\bmbz_0)^{\oplus \bfc_-(0)}\ran*\lan M(\bmbz_{-1})^{\oplus \bfc_-(-1)}\ran*\lan M(\bmbz_{-2})^{\oplus \bfc_-(-2)}\ran*\dots\\
    &=E_{\bfi_0}^{(\bfc_-(0))}*T_{\bfi_0}^{-1}(E_{\bfi_{-1}}^{(\bfc_-(-1))})*T_{\bfi_0}^{-1}T_{\bfi_{-1}}^{-1}(E_{\bfi_{-2}}^{(\bfc_-(-2))})*\dots
\end{align*}
and
\begin{align*}
    E(\bfc_+)&=\dots*\lan M(\bmbz_3)^{\oplus \bfc_+(3)}\ran*\lan M(\bmbz_{2})^{\oplus \bfc_+(2)}\ran*\lan M(\bmbz_{1})^{\oplus \bfc_+(1)}\ran\\
    &=\dots *T_{\bfi_1}T_{\bfi_{2}}(E_{\bfi_{3}}^{(\bfc_+(3))})*T_{\bfi_1}(E_{\bfi_{2}}^{(\bfc_+(2))})*E_{\bfi_1}^{(\bfc_+(1))}
\end{align*}
for any $\bfc_-\in\cG_-,\bfc_+\in\cG_+$.
Then
\begin{align*}
   E(\bfc,t_\lz)=&\left(  E_{\bfi_0}^{(\bfc_-(0))}*T_{\bfi_0}^{-1}(E_{\bfi_{-1}}^{(\bfc_-(-1))})*T_{\bfi_0}^{-1}T_{\bfi_{-1}}^{-1}(E_{\bfi_{-2}}^{(\bfc_-(-2))})*\dots\right)*E(\bfc_0,t_\lz)\\
   &*\left(\dots *T_{\bfi_1}T_{\bfi_{2}}(E_{\bfi_{3}}^{(\bfc_+(3))})*T_{\bfi_1}(E_{\bfi_{2}}^{(\bfc_+(2))})*E_{\bfi_1}^{(\bfc_+(1))}\right)
\end{align*}
for any $(\bfc,t_\lz)\in\cG^a(\ggz)$.

For $(\bfc,t_\lz)\in\cG^a(\ggz)$ and $p\in\bbZ$, we define
\begin{align*}
   E_p(\bfc,t_\lz)=&\left(  E_{\bfi_p}^{(\bfc(p))}*T_{\bfi_p}^{-1}(E_{\bfi_{p-1}}^{(\bfc(p-1))})*T_{\bfi_p}^{-1}T_{\bfi_{p-1}}^{-1}(E_{\bfi_{p-2}}^{(\bfc(p-2))})*\dots\right)\\
   &*T_{\bfi_{p+1}}T_{\bfi_{p+2}}\dots T_{\bfi_0} (E(\bfc_0,t_\lz))\\
   &*\left(\dots *T_{\bfi_{p+1}}T_{\bfi_{p+2}}(E_{\bfi_{p+3}}^{(\bfc(p+3))})*T_{\bfi_{p+1}}(E_{\bfi_{p+2}}^{(\bfc(p+2))})*E_{\bfi_{p+1}}^{(\bfc(p+1))}\right)
\end{align*}
for $p\leq 0$. For $p>0$, we replace the middle part by $T^{-1}_{\bfi_{p}}T^{-1}_{\bfi_{p-1}}\dots T^{-1}_{\bfi_1} (E(\bfc_0,t_\lz))$. 
Here $\bfc=\bfc_-+\bfc_+$ as functions on $\bbZ$.
In particular, $E_0(\bfc,t_\lz)=E(\bfc,t_\lz)$.

The following lemma is obtained by observation.
\begin{lemma}
    For $p\in\bbZ$, the set $\{E_p(\bfc,t_\lz)|(\bfc,t_\lz)\in\cG^a(\ggz)\}$ is the PBW basis for $\sz_{\bfi_{p}}\dots\sz_{\bfi_{-1}}\sz_{\bfi_0}\ggz$ ($p\leq 0$), or respectively, $\sz_{\bfi_{p}}\dots\sz_{\bfi_{2}}\sz_{\bfi_1}\ggz$ ($p>0$), with the admissible sequence $\underline{\bfi}[p]$.
    Here $\underline{\bfi}[p]=(\dots,\bfi_{p-1}|\bfi_p,\bfi_{p+1},\dots)$.
\end{lemma}

Denote by $C_p(\bfc,t_\lz)$ the unique element in $\cB$ such that $C_p(\bfc,t_\lz)\equiv E_p(\bfc,t_\lz)\,\mod v^{-1}\cL_{\cA'}$. In particular, $C_0(\bfc,t_\lz)=C(\bfc,t_\lz)$. 
Then for any fixed $p\in\bbZ$, $\cB=\{\pm C_p(\bfc,t_\lz)|(\bfc,t_\lz)\in\cG^a(\ggz)\}$.

\begin{proposition}\label{k(C)=C_0}
    For $(\bfc,t_\lz)\in\cG^a(\ggz)$, if $\bfc_-(0)=0$, then $$\kappa_{\bfi_0}(C(\bfc,t_\lz))=C_{-1}(\bfc,t_\lz).$$
\end{proposition}

\begin{proof}
    We already showed that $C(\bfc,t_\lz)\in\cB_{\bfi_0;0}=\cB-\cB\cap E_{\bfi_0}\bfU^+$ for $\bfc_-(0)=0$ in Proposition \ref{prop: pi acts on the left-most side}.
    By Lemma \ref{b-pi(b)}, 
    $${}^{\bfi_0}\pi(C(\bfc,t_\lz)))\equiv C(\bfc,t_\lz)\equiv E(\bfc,t_\lz)\,\mod v^{-1}\bbQ\otimes_\bbZ\cL_{\cA}.$$
    Note that $E(\bfc,t_\lz)\in\bfU^+[{\bfi_0}]$ for $\bfc_-(0)=0$, then by applying $T_{\bfi_0}$ to both sides, since $T_{\bfi_0}$ maps $\cL_\cA\cap\bfU^+[\bfi_0]$ into itself and $\bbQ\otimes_\bbZ\cL_\cA\subset\cL_{\cA'}$, we have
    $$T_{\bfi_0}({}^{\bfi_0}\pi(C(\bfc,t_\lz)))\equiv E_{-1}(\bfc,t_\lz)\,\mod v^{-1}\cL_{\cA'}.$$
    By Lemma \ref{b-pi(b)}, Lemma \ref{kappa is T},
    $$T_{\bfi_0}({}^{\bfi_0}\pi(C(\bfc,t_\lz)))=\pi^{\bfi_0}(\kappa_{\bfi_0}(C(\bfc,t_\lz)))\equiv \kappa_{\bfi_0}(C(\bfc,t_\lz))\,\mod v^{-1}\cL_{\cA'}.$$
    So $\kappa_{\bfi_0}(C(\bfc,t_\lz))\equiv  E_{-1}(\bfc,t_\lz)\,\mod v^{-1}\cL_{\cA'}$.
    Then $\kappa_{\bfi_0}(C(\bfc,t_\lz))=C_{-1}(\bfc,t_\lz)$ by definition.
\end{proof}

The following corollary is naturally deduced by Proposition \ref{k(C)=C_0}.
\begin{corollary}\label{k(C)=C_s}
    For $(\bfc,t_\lz)\in\cG^a(\ggz)$, let $s\leq 0$ such that $\bfc_-(t)=0$ for all $s\leq t\leq 0$, then $$\kappa_{\bfi_{s}}\dots\kappa_{\bfi_{-1}}\kappa_{\bfi_0}(C(\bfc,t_\lz))=C_{s-1}(\bfc,t_\lz).$$
\end{corollary}

\subsection{About the regular simples}

Assume in this section that $\ggz$ is of type $\tilde{B}_n,\tilde{C}_n$.
Fix a finite field $k$ and a non-homogeneous tube, for example, $\cT=\cT_1$.
\begin{lemma}
    If $\ggz$ is of type $\tilde{B}_n,\tilde{C}_n$, then for any regular simple module $M$, there exists $p\leq 0$ such that $\sz^+_{\bfi_{p+1}}\sz^+_{\bfi_{p+2}}\dots\sz^+_{\bfi_{0}} M$ is a simple module in $\mod_k\ggz$.
\end{lemma}

\begin{proof}
    By \cite{Dlab_Ringel_Indecomposable_representations_of_graphs_and_algebras}, we can choose an orientation for $\tilde{B}_n$ or $\tilde{C}_n$ such that there is at least one simple module $M$ in the unique non-homogeneous regular tube. Then other non-homogeneous regular simple modules can be obtained by applying Auslandar-Reiten translation functor by several times on $M$. Since AR translation functor can be written as a composition of some BGP reflection functors by \cite{Bernstein_Gel'fand_Ponomarev_Coxeter_functors_and_Gabriel's_theorem}, the lemma is true for this chosen orientation. Applying BGP reflection functors will prove this lemma for any orientation.
\end{proof}

Now we generalize Proposition \ref{C(pi)=phi C(pi|)}. 
Let $i$ be any vertex in the quiver $K$, then $\vez(S_i)$ is regular simple.
By lemma above, denote $\vez(S_i)=\sz^-_{\bfi_{0}}\dots\sz^-_{\bfi_{p+2}}\sz^-_{\bfi_{p+1}}(S_\bfj)$ for some $p\leq 0$ and $\bfj\in\bfI$.
\begin{proposition}\label{generalization of C(pi)=phi C(pi|)}
    Fix $(\pi_2,\dots,\pi_s)\in \Pi^a_{r_2}\times \dots\times\Pi^a_{r_s}$ and $t_\lz\in\bbP$.
    For $\pi\in\Pi^a$, we have
    $$C((\pi\uparrow_i,\pi_2,\dots,\pi_s),t_\lz)=\kappa^{-1}_{\bfi_0}\dots\kappa^{-1}_{\bfi_{p+2}}\kappa^{-1}_{\bfi_{p+1}}\varphi_{\bfj}\kappa_{\bfi_{p+1}}\kappa_{\bfi_{p+2}}\dots\kappa_{\bfi_0} C((\pi,\pi_2,\dots,\pi_s),t_\lz).$$
\end{proposition}
  
\begin{proof}
    This follows directly from Proposition \ref{C(pi)=phi C(pi|)} and Corollary \ref{k(C)=C_s}.
\end{proof}

\subsection{About $sl_2$}
Also assume that $\ggz$ is of type $\tilde{B}_n,\tilde{C}_n$.
Let $\ooz$ be the Kronecker quiver.
There are two vertices, denoted by $1,2$, and two arrows, both from $1$ to $2$, in $\ooz$.
Let $\llz_k$ be the representation category of $\ooz$ over the finite field $k$.

For $\ggz$ of type $\tilde{B}_n$, there is 
indecomposable $M_1\in\cI$ and $M_2\in\cP$ such that $\udim M_1+\udim M_2=\dz$ and $$\End_kM_1=\End_kM_2=1,\Hom_k(M_1,M_2)=0,\Ext^1_k(M_1,M_2)=2.$$
Then $\llz_k$ can be embedded into $\Rep_k\ggz$ by mapping $S_1,S_2$ to $M_1,M_2$ respectively.

For $\ggz$ of type $\tilde{C}_n$, let $k'$ be a finite field such that $[k':k]=2$. There is 
indecomposable $M_1\in\cI$ and $M_2\in\cP$ such that $\udim M_1+\udim M_2=\dz$ and $$\End_kM_1=\End_kM_2=2,\Hom_k(M_1,M_2)=0,\Ext^1_k(M_1,M_2)=4.$$
Then $\llz_{k'}$ can be embedded into $\Rep_k\ggz$ by mapping $S_1,S_2$ to $M_1,M_2$ respectively.

For both cases, denote 
$$M_1=\sz^+_{\bfi_1}\sz^+_{\bfi_2}\dots\sz^+_{\bfi_{r-1}}(S_{\bfi_r}), M_2=\sz^-_{\bfi_0}\sz^-_{\bfi_{-1}}\dots\sz^-_{\bfi_{l+1}}(S_{\bfi_{l}}), r\geq 1, l\leq 0.$$

Consider the generic composition algebra $\cC^*(\ooz)$. By \cite{XXZ_Tame_quivers_and_affine_bases_I}, $\cC^*(\ooz)$ has a canonical basis
$$\bfB^\ooz=\{B^\ooz(\bfc_-,\bfc_+,t_\lz)|\bfc_-\in\cG^-,\bfc_+\in\cG^+,t_\lz\in\bbP\}$$
such that 
$$B^\ooz(\bfc_-,\bfc_+,t_\lz)\equiv N^\ooz(\bfc_-,\bfc_+,t_\lz)=\lan M^\ooz(\bfc_-)\ran*S^\ooz_\lz*\lan M^\ooz(\bfc_+)\ran\,\mod v^{-1}\cL^\ooz_{\cA}.$$
Let $\cB^\ooz=\bfB^\ooz\sqcup(-\bfB^\ooz)$.
Then the subsets $\cB^\ooz_{j;n}$, the bijections $\varphi^\ooz_{j;n}:\cB^\ooz_{j;0}\ra\cB^\ooz_{j;n}$ and the maps $\varphi^\ooz_{j}:\cB^\ooz\ra\cB^\ooz$ are well defined.

By abuse of notation, denote $\iota:\cC^*(\ooz)\ra\cH^0(\ggz)$ the injection induced by the embedding $\iota$ of $\llz_k$ or $\llz_{k'}$ into $\Rep_k\ggz$. Denote also $\iota:\cG^a(\ooz)\ra\cG^a(\ggz)$ such that $\iota(\bfc,t_\lz)=(\bfd,t_\lz)$ where $\iota(M(\bfc))\cong M(\bfd)$.

For any $\az\in\cG^a(\ooz)$, denote by $\az\uparrow_j\in\cG^a(\ooz), j=1,2$ such that 
$$\varphi^\ooz_j(B^oz(\az))=B^oz(\az\uparrow_j).$$

\begin{proposition}\label{Kash in sl2}
   For $\az\in\cG^a(\ooz)$, we have 
   $$\kappa_{\bfi_1}\kappa_{\bfi_2}\dots\kappa_{\bfi_{r-1}}\varphi_{\bfi_r}\kappa_{\bfi_{r-1}}^{-1}\dots\kappa_{\bfi_2}^{-1}\kappa_{\bfi_1}^{-1} C(\iota(\az))=C(\iota(\az\uparrow_1)),$$
   $$\kappa_{\bfi_0}^{-1}\kappa_{\bfi_{-1}}^{-1}\dots\kappa_{\bfi_{l+1}}^{-1}\varphi_{\bfi_l}\kappa_{\bfi_{l+1}}\dots\kappa_{\bfi_{-1}}\kappa_{\bfi_0} C(\iota(\az))=C(\iota(\az\uparrow_2)).$$
\end{proposition}
\begin{proof}
    The proof is similar to the proof of Proposition \ref{C(pi)=phi C(pi|)}, where we need only to check that the Kashiwara operators $\tphi_j$ commute with the embeddings if $\iota(S_j)$ is simple, and then by Corollary \ref{k(C)=C_s} the proof is done.
\end{proof}

\section{Remove the sign for type $\tilde{B}_n$, $\tilde{C}_n$}\label{sec remove sgn}

We write $\Delta(\bfc,t_\lz)=\sgn(C(\bfc,t_\lz))$ for short.

\begin{lemma}\label{lem1}
For $(\bfc,t_\lz)\in\cG^a$, $\Delta(\bfc,t_\lz)=\Delta(\bfc_0,t_\lz)$.
\end{lemma}

\begin{proof}
    By Proposition \ref{prop: pi acts on the left-most side}, $\Delta(\bfc,t_\lz)=\Delta(\bfc^{\vdash},t_\lz)$. 
    Since $\bfc^\vdash_-(0)=0$, by Proposition \ref{k(C)=C_0}, $\kappa_{\bfi_0}(C(\bfc^\vdash,t_\lz))=C_{-1}(\bfc^\vdash,t_\lz)$, then $$\sgn(C(\bfc,t_\lz))=\sgn(C_{-1}(\bfc^\vdash,t_\lz))$$ by Lemma \ref{kappa is T}.
    Continue this process, then we will get
    $$\sgn(C(\bfc,t_\lz))=\sgn(C_{s}((0,\bfc_0,\bfc_+),t_\lz))$$ for $s\ll 0$.
    By Corollary \ref{k(C)=C_s}, $$\sgn(C(\bfc,t_\lz))=\sgn(\kappa_{\bfi_{s+1}}\dots\kappa_{\bfi_{-1}}\kappa_{\bfi_0}C((0,\bfc_0,\bfc_+),t_\lz))=\sgn (C((0,\bfc_0,\bfc_+),t_\lz)).$$
    So $\Delta(\bfc,t_\lz)=\Delta((0,\bfc_0,\bfc_+),t_\lz)$.
    Dually, $\Delta(\bfc,t_\lz)=\Delta((\bfc_-,\bfc_0,0),t_\lz)$ for any $(\bfc,t_\lz)\in\cG^a$. 
    Then $\Delta(\bfc,t_\lz)=\Delta((0,\bfc_0,\bfc_+),t_\lz)=\Delta((0,\bfc_0,0),t_\lz)$.
\end{proof}

\begin{lemma}\label{lem2}
    If $\ggz$ is of type $\tilde{B}_n,\tilde{C}_n$,
    for $(\bfc_0,t_\lz)\in\cG^a_0\times\bbP$, $\Delta(\bfc_0,t_\lz)=\Delta(0,t_\lz)$.
\end{lemma}

\begin{proof}
    By Proposition \ref{generalization of C(pi)=phi C(pi|)} and Lemma \ref{kappa is T}, we have 
    $$\Delta((\pi,\pi_2,\dots,\pi_s),t_\lz)=\Delta ((\pi\uparrow_i,\pi_2,\dots,\pi_s),t_\lz)$$ for all $i$ and $\pi\in\Pi^a$.
    Since $\bfB^K=\{B^K(\pi)|\pi\in\Pi^a\}$, by Lusztig, for any $\pi\in\Pi^a$, there is a sequence $i_1,i_2,\dots,i_m$ of vertices in $K$ such that $\pi=0\uparrow_{i_1}\uparrow_{i_2}\dots\uparrow_{i_m}$, then 
    $$\Delta((\pi,\pi_2,\dots,\pi_s),t_\lz)=\Delta ((0,\pi_2,\dots,\pi_s),t_\lz).$$
    The lemma follows by continuing this process for all non-homogeneous tubes.
\end{proof}

\begin{lemma}\label{lem3}
     If $\ggz$ is of type $\tilde{B}_n,\tilde{C}_n$, for $t_\lz\in\bbP$, $\Delta(0,t_\lz)=1$.
\end{lemma}

\begin{proof}
    By Proposition \ref{Kash in sl2}, $\Delta(\iota(\az))=\Delta(\iota(\az\uparrow_j))$ for all $j=1,2$ and $\az\in\cG^a(\ooz)$.
    Since $\bfB^\ooz=\{B^\ooz(\az)|\az\in\cG^a(\ooz)\}$, by Lusztig, for any $\az\in\cG^a(\ooz)$, there is a sequence $i_1,i_2,\dots,i_m$ in $\{1,2\}$ such that $\az=0\uparrow_{i_1}\uparrow_{i_2}\dots\uparrow_{i_m}$. 
    Then $\Delta(\iota(\az))=\Delta(0)=1$ for all $\az\in\cG^a(\ooz)$. 
    In particular, $\Delta(0,t_\lz)=\Delta(\iota(0,t_\lz))=1$.
\end{proof}

\begin{theorem}
    If $\ggz$ is of type $\tilde{B}_n,\tilde{C}_n$, for $(\bfc,t_\lz)\in\cG^a$, $\Delta(\bfc,t_\lz)=1$.
\end{theorem}
\begin{proof}
    This follows directly from Lemma \ref{lem1}, Lemma \ref{lem2}, Lemma \ref{lem3}.
\end{proof}

The following corollary is immediately obtained.
\begin{corollary}
     If $\ggz$ is of type $\tilde{B}_n,\tilde{C}_n$, for $(\bfc,t_\lz)\in\cG^a$, we have 
     $$\Xi(C(\bfc,t_\lz))=B(\bfc,t_\lz).$$
\end{corollary}

%\begin{remark}
%By listing all possible (acyclic) tame quivers and their non-trivial automorphisms and consider the Frobenius foldings on $\cR_{\rm nh}$, we can check that for all cases satisfying (3), we have $\rk T=2\rk \cT$.
%\end{remark}

%\section{Appendix}

%graphs of affine type

%This section is the main part where we need different strategy to construct monomials from the quiver cases. We already know that each nonhomogeneous tube $\cT$ in $\cR_{\rm nh}^\ggz$ is obtained by Froubenius-folding some nonhomogeneous tubes over $\bar{k}$ in $\cR_{\rm nh}$.

%There are three possible cases:
%\begin{enumerate}
%\item[(1)] $\cT$ is the Froubenius-folding of $T$ where $T$ is a nonhomogeneous tube in $\cR_{\rm nh}$ and $(-)^{[1]}$ acts identically on $T$. This is the trivial case.
%\item[(2)] $\cT$ is the Froubenius-folding of $T,T^{[1]},\cdots,T^{[l-1]}$, where $T$ is a nonhomogeneous tube in $\cR_{\rm nh}$ and $l> 1$ is minimal such that $T^{[l]}=T$. In this case, $\rk T=\rk \cT$.
%\item[(3)] $\cT$ is the Froubenius-folding of a single nonhomogeneous tube $T$ in $\cR_{\rm nh}$, but $(-)^{[1]}$ acts on $T$ by a nontrivial rotation. In this case, $\rk T=p\rk \cT$ for some $p>1$.
%\end{enumerate}

\bibliography{Tame_quivers_and_affine_bases_II}

\begin{thebibliography}{10}

\bibitem{Bernstein_Gel'fand_Ponomarev_Coxeter_functors_and_Gabriel's_theorem}
I.~N. Bern\v{s}te\u{\i}n, I.~M. Gel'fand, and V.~A. Ponomarev.
\newblock Coxeter functors, and {G}abriel's theorem.
\newblock {\em Uspehi Mat. Nauk}, 28(2(170)):19--33, 1973.

\bibitem{BBD_Faisceaux_pervers}
A.~A. Be\u{\i}linson, J.~Bernstein, and P.~Deligne.
\newblock Faisceaux pervers.
\newblock In {\em Analysis and topology on singular spaces, {I} ({L}uminy, 1981)}, volume 100 of {\em Ast\'{e}risque}, pages 5--171. Soc. Math. France, Paris, 1982.

\bibitem{Bongartz_On_degenerations_and_extensions_of_finite-dimensional_modules}
K.~Bongartz.
\newblock On degenerations and extensions of finite-dimensional modules.
\newblock {\em Adv. Math.}, 121(2):245--287, 1996.

\bibitem{Deng_Du_Frobenius_morphisms_and_representations_of_algebras}
B.~Deng and J.~Du.
\newblock Frobenius morphisms and representations of algebras.
\newblock {\em Transactions of the American Mathematical Society}, 358(8):3591--3622, 2006.

\bibitem{DDPW_Finite_dimensional_algebras_and_quantum_groups}
B.~Deng, J.~Du, B.~Parshall, and J.~Wang.
\newblock {\em Finite dimensional algebras and quantum groups}, volume 150 of {\em Mathematical Surveys and Monographs}.
\newblock American Mathematical Society, Providence, RI, 2008.

\bibitem{Deng_Du_Xiao_Generic_extensions_and_canonical_bases_for_cyclic_quivers}
B.~Deng, J.~Du, and J.~Xiao.
\newblock Generic extensions and canonical bases for cyclic quivers.
\newblock {\em Canad. J. Math.}, 59(6):1260--1283, 2007.

\bibitem{Deng_Han}
B.~Deng and L.~Han.
\newblock Hall polynomials for tame quivers with automorphism.
\newblock {\em Journal of Algebra}, 2022.

\bibitem{Deng_Xiao_Ringel-Hall_algebras_and_Lusztig's_symmetries}
B.~Deng and J.~Xiao.
\newblock Ringel-{H}all algebras and {L}usztig's symmetries.
\newblock {\em J. Algebra}, 255(2):357--372, 2002.

\bibitem{Dlab_Ringel_Indecomposable_representations_of_graphs_and_algebras}
V.~Dlab and C.~M. Ringel.
\newblock Indecomposable representations of graphs and algebras.
\newblock {\em Mem. Amer. Math. Soc.}, 6(173):v+57, 1976.

\bibitem{Green_Hall_algebras_hereditary_algebras_and_quantum_groups}
J.~A. Green.
\newblock Hall algebras, hereditary algebras and quantum groups.
\newblock {\em Invent. Math.}, 120(2):361--377, 1995.

\bibitem{Li_Notes_on_affine_canonical_and_monomial_bases}
Y.~Li.
\newblock Notes on affine canonical and monomial bases.
\newblock {\em arXiv preprint arXiv:math/0610449}, 2006.

\bibitem{Lusztig_Affine_quivers_and_canonical_bases}
G.~Lusztig.
\newblock Affine quivers and canonical bases.
\newblock {\em Inst. Hautes \'{E}tudes Sci. Publ. Math.}, (76):111--163, 1992.

\bibitem{Lusztig_Braid_group_action_and_canonical_bases}
G.~Lusztig.
\newblock Braid group action and canonical bases.
\newblock {\em Adv. Math.}, 122(2):237--261, 1996.

\bibitem{Lusztig_Canonical_bases_and_Hall_algebras}
G.~Lusztig.
\newblock Canonical bases and {H}all algebras.
\newblock In {\em Representation theories and algebraic geometry ({M}ontreal, {PQ}, 1997)}, volume 514 of {\em NATO Adv. Sci. Inst. Ser. C Math. Phys. Sci.}, pages 365--399. Kluwer Acad. Publ., Dordrecht, 1998.

\bibitem{Lusztig_Introduction_to_quantum_groups}
G.~Lusztig.
\newblock {\em Introduction to quantum groups}.
\newblock Modern Birkh\"{a}user Classics. Birkh\"{a}user/Springer, New York, 2010.
\newblock Reprint of the 1994 edition.

\bibitem{Macdonald_Symmetric_functions_and_Hall_polynomials}
I.~G. Macdonald.
\newblock {\em Symmetric functions and {H}all polynomials}.
\newblock Oxford Classic Texts in the Physical Sciences. The Clarendon Press, Oxford University Press, New York, second edition, 2015.
\newblock With contribution by A. V. Zelevinsky and a foreword by Richard Stanley, Reprint of the 2008 paperback edition [MR1354144].

\bibitem{Ringel_Representations_of_K-species_and_bimodules}
C.~M. Ringel.
\newblock Representations of k-species and bimodules.
\newblock {\em Journal of Algebra}, 41(2):269--302, 1976.

\bibitem{Ringel_Hall_algebras_and_quantum_groups}
C.~M. Ringel.
\newblock Hall algebras and quantum groups.
\newblock {\em Invent. Math.}, 101(3):583--591, 1990.

\bibitem{Ringel_PBW-bases_of_quantum_groups}
C.~M. Ringel.
\newblock P{BW}-bases of quantum groups.
\newblock {\em J. Reine Angew. Math.}, 470:51--88, 1996.

\bibitem{Ringel_Green's_Theorem_on_Hall_Algebras}
C.~M. Ringel.
\newblock Green's theorem on hall algebras.
\newblock 1999.

\bibitem{XXZ_Tame_quivers_and_affine_bases_I}
J.~Xiao, H.~Xu, and M.~Zhao.
\newblock Tame quivers and affine bases {I}: A {H}all algebra approach to the canonical bases.
\newblock {\em Journal of Algebra}, 633:510--562, 2023.

\end{thebibliography}

\end{document}